\documentclass[11pt,twoside,reqno]{amsart}

\usepackage{amssymb, latexsym}
\usepackage[italian,english]{babel}
\usepackage{amsmath,amsfonts,amsthm}
\usepackage{paralist}
\usepackage[top=2.50cm, bottom=2.50cm, left=2.50cm, right=2.50cm]{geometry}

\usepackage{color}

\numberwithin{equation}{section}

\usepackage{xspace}
\usepackage[bookmarksnumbered,colorlinks]{hyperref}
\usepackage{graphics}
\def\bb#1\eb{\textcolor{blue}
{#1}} %
\def\br#1\er{\textcolor{red}
{#1}} %
\def\bv#1\ev{\textcolor{green}
{#1}} %
\def\bc#1\ec{\textcolor{cyan}
{#1}} %

\usepackage{graphics}

\def\Xint#1{\mathchoice
  {\XXint\displaystyle\textstyle{#1}}%
  {\XXint\textstyle\scriptstyle{#1}}%
  {\XXint\scriptstyle\scriptscriptstyle{#1}}%
  {\XXint\scriptscriptstyle\scriptscriptstyle{#1}}%
  \!\int}
\def\XXint#1#2#3{{\setbox0=\hbox{$#1{#2#3}{\int}$}
  \vcenter{\hbox{$#2#3$}}\kern-.5\wd0}}
\def\-int{\Xint -}

\newcommand{\R}{\mathbb{R}}
\newcommand{\N}{\mathbb{N}}

\DeclareMathOperator{\I}{\mathcal{I}}

\newtheorem{lem}{Lemma}[section]
\newtheorem{thm}{Theorem}[section]

\newtheorem{remark}{Remark}[section]

\begin{document}
\title[Ground state solutions]{Ground state solutions for a fractional Schr\"odinger equation with critical growth}

\author[V. Ambrosio]{Vincenzo Ambrosio}
\address{Dipartimento di Scienze Pure e Applicate (DiSPeA) \\
Universit\`a degli Studi di Urbino 'Carlo Bo' \\
Piazza della Repubblica, 13  \\
61029 Urbino (Pesaro e Urbino, Italy)}
\email{vincenzo.ambrosio@uniurb.it}
\author[G. M. Figueiredo]{Giovany M. Figueiredo}
\address{Departamento de Matem\'atica\\
Universidade de Bras\'ilia \\
70.910-900, Bras\'ilia - DF, Brazil}
\email{giovany@unb.br}



\begin{abstract}
In this paper we investigate the existence of nontrivial ground state solutions for the following fractional scalar field equation
\begin{align*}
(-\Delta)^{s} u+V(x)u= f(u) \mbox{ in } \R^{N}, 
\end{align*}
where $s\in (0,1)$, $N> 2s$, $(-\Delta)^{s}$ is the fractional Laplacian, $V: \R^{N}\rightarrow \R$ is a bounded potential satisfying suitable assumptions, and $f\in C^{1, \beta}(\R, \R)$ has critical growth. We first analyze the case $V$ constant, and then we develop a Jeanjean-Tanaka argument \cite{JT} to deal with the non autonomous case. 
As far as we know, all results presented here are new.
\end{abstract}

\maketitle
\section{Introduction}

\noindent
This paper is devoted to the existence of nontrivial solutions for the following fractional scalar field equation 
\begin{align}\label{P}
(-\Delta)^{s} u+V(x)u= f(u) \mbox{ in } \R^{N} 
\end{align}
with $s\in (0,1)$, $N> 2s$, $V:\R^{N}\rightarrow \R$ is a continuous function, and $f: \R\rightarrow \R$ is a smooth function verifying some suitable growth conditions. The fractional Laplacian $(-\Delta)^{s}$ is a pseudo-differential operator defined via Fourier transform by 
$$
\mathcal{F}(-\Delta)^{s}u(\xi)=|\xi|^{2s} \mathcal{F}u(\xi) \quad (\xi\in \R^{N}),
$$
when $u: \R^{N}\rightarrow \R$ belongs to the Schwarz space $\mathcal{S}(\R^{N})$ of rapidly decaying functions.
Also, $(-\Delta)^{s}u$ can be equivalently represented, up to normalization factors, as
$$
(-\Delta)^{s}u(x)=-\int_{\R^{N}} \frac{u(x+y)+u(x-y)-2u(x)}{|y|^{N+2s}} dy  \quad (x\in \R^{N});
$$
see for instance \cite{DPV} for more details.\\
The main motivation of the study of (\ref{P}) comes from looking for standing waves $\psi(x, t)=u(x)e^{-\imath ct}$ for the fractional Schr\"odinger equation 
$$
\imath \frac{\partial \psi}{\partial t}=(-\Delta)^{s} \psi+V(x)\psi -f(\psi) \quad (t, x)\in \R\times \R^{N}.
$$
Such equation has been introduced by Laskin \cite{Laskin1, Laskin2}, as a result of expanding the Feynman path integral, from the Brownian like to the L\'evy like quantum mechanical paths.
\noindent

When $s=1$ in (\ref{P}), we derive the classical nonlinear Schr\"odinger equation which has been extensively studied in the last twenty years by many authors. Since we cannot review the huge bibliography here, we just mention the works \cite{AS, AM2, BL1, DF, FW, Rab, Wang, W} and references therein, where several results on the existence and the multiplicity of solutions are obtained under different assumptions on the potential $V$ and the nonlinearity $f$.
\noindent

In these last years, problems involving fractional operators are receiving a lot of attention.
Indeed fractional spaces and nonlocal equations play a fundamental role in the investigation of many several sciences such as crystal dislocation, 
obstacle problem, optimization, finance, phase transition, soft thin films, multiple scattering, quasi-geostrophic flows, water waves, anomalous diffusion, conformal geometry and minimal surfaces and so on.
The interested reader may consult \cite{DPV} and \cite{MBRS}, where a more extensive bibliography and an introduction to the subject are given.

\noindent
In what follows, we recall some known results established in recent years, concerning with the fractional Laplacian equations with critical growth.

Servadei and Valdinoci \cite{SV} (see also \cite{FMS, MS}) established a Brezis-Nirenberg type result for the following problem
\begin{equation*}
\left\{
\begin{array}{ll}
(-\Delta)^{s}u-\lambda u=|u|^{2^{*}_{s}-2}u &\mbox{ in } \Omega   \\
u=0   &\mbox{ in }  \R^{N}\setminus \Omega,
\end{array}
\right.
\end{equation*}
where $\Omega\subset \R^{N}$ is a smooth bounded domain and $\lambda>0$ is a parameter.
Barrios at al. \cite{BCPS} studied the effect of lower order perturbations in the existence of positive solutions to the following critical elliptic problem involving the spectral Laplacian
\begin{equation*}
\left\{
\begin{array}{ll}
(-\Delta_{\Omega})^{s}u-\lambda u^{q}=u^{2^{*}_{s}-1} &\mbox{ in } \Omega   \\
u=0   &\mbox{ on }  \partial \Omega,
\end{array}
\right.
\end{equation*}
where $q\in (0, 2^{*}_{s}-1)$ and $\lambda>0$; see also \cite{BCSS, CaT, tan} for related results.
Fiscella and Valdinoci \cite{FV} dealt with the existence and the asymptotic behavior of non-negative solutions for a class of stationary Kirchhoff  problems driven by the fractional Laplacian 
\begin{equation*}
\left\{
\begin{array}{ll}
-M\left(\iint_{\R^{2N}} \frac{|u(x)-u(y)|^{2}}{|x-y|^{N+2s}} \, dx dy\right)(-\Delta)^{s}u=\lambda f(x, u)+|u|^{2^{*}_{s}-2}u &\mbox{ in } \Omega   \\
u=0   &\mbox{ in }  \R^{N}\setminus \Omega,
\end{array}
\right.
\end{equation*}
where $M$ is a Kirchhoff function and $f$ satisfies the Ambrosetti-Rabinowitz condition.\\
By using variational methods, Shang and Zhang \cite{SZ} studied the existence and the multiplicity of nonnegative solutions for
\begin{align*}
\varepsilon^{2s}(-\Delta)^{s} u+V(x)u=|u|^{2^{*}_{s}-2}u+ \lambda f(u) \mbox{ in } \R^{N} 
\end{align*}
where $\varepsilon, \lambda>0$, $f$ has a subcritical growth, and $V$ is a positive continuous function such that 
$$
0<\inf_{x\in \R^{N}} V(x)<\liminf_{|x|\rightarrow \infty} V(x)<\infty.
$$
Teng and He \cite{TH} combined the $s$-harmonic extension method of Caffarelli and Silvestre \cite{CS}, the concentration-compactness principle of Lions and methods of Brezis and Nirenberg to prove the existence of ground state solutions for 
\begin{align*}
(-\Delta)^{s} u+u=P(x)|u|^{p-2}u+Q(x)|u|^{2^{*}_{s}-2}u \mbox{ in } \R^{N},
\end{align*}
where $p\in (2, 2^{*}_{s})$ and $P(x)$ and $Q(x)$ are continuous functions verifying appropriate hypotheses.\\
Zhang et al. \cite{ZZR} investigated existence of nontrivial radially symmetric solutions for
\begin{align*}
(-\Delta)^{s} u+V(x)u=k(x)f(u)+\lambda |u|^{2^{*}_{s}-2}u \mbox{ in } \R^{N} 
\end{align*}
where $V(x)$ and $k(x)$ are radially symmetric functions satisfying some extra assumptions, and the nonlinearity $f$ is subcritical.
He and Zou \cite{HZ} obtained, via penalization technique and Ljusternik-Schnirelmann theory, the existence and concentration results for the problem
\begin{equation*}
\left\{
\begin{array}{ll}
\varepsilon^{2s}(-\Delta)^{s} u+V(x)u=f(u)+u^{2^{*}_{s}-1} &\mbox{ in } \R^{N}   \\
u>0   &\mbox{ in }  \R^{N}
\end{array},
\right.
\end{equation*}
under local condition imposed on $V$, and $f$ is a subcritical nonlinearity.\\
Further results concerning the fractional Schr\"odinger equations involving critical and subcritical nonlinearities can be found in \cite{ADOM, A3, A1, A4, A2, DDDV, DPPV, DoMS, FQT, FS, MBR, Secchi1}.
\noindent
 
Inspired by the above works, in the present paper we aim to investigate the existence of least energy solutions for the equation (\ref{P}), when $f$ has a critical growth and $V$ is a bounded potential satisfying some suitable assumptions.\\
More precisely, we assume that $f: \R \rightarrow \R$ verifies the following hypotheses:
\begin{compactenum}[($f1$)]
\item $f\in C^{1, \beta}(\R, \R)$ for some $\beta>\max\{0, 1-2s\}$;
\item $\displaystyle{\lim_{t\rightarrow 0^{+}} \frac{f(t)}{t}=0}$;
\item $\displaystyle{\lim_{t\rightarrow +\infty} \frac{f(t)}{t^{2^{*}_{s}-1}}= K>0}$, where $\displaystyle{2^{*}_{s}=\frac{2N}{N-2s}}$;
\item There exist  $D>0$ and $\displaystyle{\max\left\{2, \frac{4s}{N-2s}\right\}<q<2^{*}_{s}}$ such that  
$$
f(t)\geq K t^{{2^{*}_{s}}-1}+ D t^{q-1} \mbox{ for all } t\geq 0;
$$
\item There exists $C>0$ such that $|f'(t)|\leq C(1+|t|^{2^{*}_{s}-2})$ for all $t\geq 0$.
\end{compactenum}
We observe that the assumptions $(f3)$ and $(f4)$ on the nonlinearity $f$ enable us to consider the critical growth case. In the case $s=1$, the assumption $(f4)$ was introduced in \cite{ZZC} to study a Berestycki-Lions type problem with critical growth. We point out that $(f4)$ plays an important role to ensure the existence of solutions for the problem (\ref{P}). In fact, if we take $f(t)=(t^{+})^{2^{*}_{s}-1}$, then $f$ satisfies $(f1)$-$(f3)$, and by using the Pohozaev identity \cite{A1, CW} for the fractional Laplacian, we can see that there are no  nontrivial solutions to (\ref{P}). \\ 
Concerning the potential $V: \R^{N}\rightarrow \R$, we suppose that
\begin{compactenum}[($V1$)]
\item $V\in C^{1}(\R^{N}, \R)$;
\item there exists $V_{0}>0$ such that $\inf_{x\in \R^{N}} V(x)\geq V_{0}$;
\item $V(x)\leq V_{\infty}:=\lim_{|x|\rightarrow \infty} V(x)$ for all $x\in \R^{N}$;
\item 
$|\max\{x\cdot \nabla V(x), 0\}|_{L^{\frac{N}{2s}}(\R^{N})}< 2s S_{*}$, where $S_{*}$ is the best constant  of the embedding $H^{s}(\R^{N})$ into $L^{2^{*}}(\R^{N})$ (see \cite{CT}).
\end{compactenum}

\noindent
Now we state our first main result concerning the existence of ground state solutions to (\ref{P}) in the case of constant potentials (which clearly verify the assumptions $(V1)$-$(V4)$).
\begin{thm}\label{thm1}
Let $s\in (0,1)$ and $N> 2s$. Assume that  $f$ verifies $(f1)$-$(f4)$ and $V(x)\equiv V>0$ is constant. Then $(\ref{P})$ possesses a nontrivial ground state solution $u\in H^{s}(\R^{N})$. 
\end{thm}
\noindent
Now, we give a sketch of the proof of Theorem \ref{thm1}. 
We recall that for a weak solution of problem (\ref{P}), we mean a function $u\in H^{s}(\R^{N})$ such that 
$$
\iint_{\R^{2N}} \frac{(u(x)-u(y))}{|x-y|^{N+2s}} (\varphi(x)-\varphi(y)) \, dx dy+\int_{\R^{N}} V(x) u \varphi \, dx=\int_{\R^{N}} f(u) \varphi \, dx
$$
for any $\varphi \in H^{s}(\R^{N})$.\\
Here $H^{s}(\R^{N})$ is the fractional Sobolev space defined by 
$$
H^{s}(\R^{N})=\left\{u\in L^{2}(\R^{N}): \frac{u(x)-u(y)}{|x-y|^{\frac{N+2s}{2}}}\in L^{2}(\R^{2N})\right\}.
$$
In order to obtain the existence of a nontrivial solution to (\ref{P}), we look for critical points of the Euler-Lagrange functional associated to (\ref{P}), that is
$$
\I(u)=\frac{1}{2} \int_{\R^{N}} \left(|(-\Delta)^{\frac{s}{2}} u|^{2}+V(x) u^{2}\right) \, dx-\int_{\R^{N}} F(u) \, dx
$$
for any $u\in H^{s}(\R^{N})$, where $F(t)=\int_{0}^{t} f(\tau) \, d\tau$. 
By using the assumptions on $f$, it is clear that $\I$ has a mountain pass geometry, but it is hard to verify the boundedness of Palais-Smale sequences of $\I$ (such Palais-Smale sequences there exist in view of the Ekeland's principle). To overcome this difficulty, we use the idea in \cite{J}.
For $\lambda\in [\frac{1}{2}, 1]$, let us introduce the following family of functionals
\begin{equation*}
\I_{\lambda}(u)=\frac{1}{2} \int_{\R^{N}} \left(|(-\Delta)^{\frac{s}{2}} u|^{2}+V(x) u^{2}\right) \, dx-\lambda \int_{\R^{N}} F(u) \, dx .
\end{equation*}
As first step, we prove that for any $\lambda\in [\frac{1}{2}, 1]$, $\I_{\lambda}$ has a mountain pass geometry and that $\I_{\lambda}$ admits a bounded Palais-Smale sequence $(u_{n})$ at the mountain-pass level $c_{\lambda}$. More precisely, we use the following abstract result due to Jeanjean \cite{J}:
\begin{thm}\cite{J}\label{thm2.1}
Let $(X, \|\cdot\|)$ be a Banach space and $J \subset \R_{+}$ be an interval.
Let $(\I_{\lambda})_{\lambda\in J}$ be a family of $C^{1}$ functionals on $X$ of the form
$$
\I_{\lambda}(u)=A(u)-\lambda B(u), \quad \mbox{ for } \lambda \in J,
$$
where $B(u)\geq 0$ for all $u\in X$, and either $A(u)\rightarrow \infty$ or $B(u)\rightarrow \infty$ as $\|u\|\rightarrow \infty$. \\
We assume that there exist $v_{1}, v_{2}\in X$ such that
$$
c_{\lambda}=\inf_{\gamma \in \Gamma} \max_{t\in [0, 1]} \I_{\lambda}(\gamma(t))>\max\{\I_{\lambda}(v_{1}), \I_{\lambda}(v_{2})\}, \quad \forall \lambda \in J
$$
where 
$$
\Gamma=\{\gamma\in C([0, 1], X): \gamma(0)=v_{1}, \gamma(1)=v_{2}\}.
$$
Then, for almost every $\lambda\in J$, there is a sequence $(u_{n})\subset X$ such that
\begin{compactenum}[(i)]
\item $(u_{n})$ is bounded;
\item $\I_{\lambda}(u_{n})\rightarrow c_{\lambda}$;
\item $\I'_{\lambda}(u_{n})\rightarrow 0$ on $X^{-1}$.
\end{compactenum}
Moreover, the map $\lambda \mapsto c_{\lambda}$ is continuous from the left hand-side.
\end{thm}

\noindent
Since we are dealing with the critical case, we are able to prove that for any $\lambda\in [\frac{1}{2}, 1]$ 
$$
0<c_{\lambda}<\frac{s}{N}\frac{S_{*}^{\frac{N}{2s}}}{\lambda^{\frac{N-2s}{2s}}}.
$$
Secondly, in the spirit of \cite{JT} (see also \cite{LG, ZZ}), we establish a global compactness result in the critical case, which gives a description of the bounded Palais-Smale sequences of $\I_{\lambda}$. Then, by using the facts that every solution of (\ref{P}) satisfies the Pohozaev Identity  and the compactness Lemma, we prove the existence of a bounded Palais-Smale sequence of $\I$ which converges to a positive solution to (\ref{P}).

\noindent

Now, we state our second main result of this paper, which deals with the existence of ground state of (\ref{P}) in the case in which $V$ is not a constant. 
\begin{thm}\label{thm2}
Let $s\in (0,1)$ and $N> 2s$. Assume that  $f$ verifies $(f1)$-$(f5)$ and $V$ satisfies $(V1)$-$(V4)$, and $V(x)\not \equiv V_{\infty}$. Then $(\ref{P})$ admits a nontrivial ground state solution $u\in H^{s}(\R^{N})$.
\end{thm}

\noindent
To deal with the non-autonomous case, we resemble some ideas developed in \cite{JT}. We consider the previous  family of functionals $\I_{\lambda}$, and, since $\I_{\lambda}$ satisfies the assumptions of Theorem \ref{thm2.1}, we can deduce the existence of a Palais-Smale sequence $(u^{j}_{n})$ at the mountain-pass level $c_{\lambda_{j}}$, where $\lambda_{j}\rightarrow 1$. Therefore, $u_{n}^{j}\rightharpoonup u_{j}$ in $H^{s}(\R^{N})$ where $u_{j}$ is a critical point of $\I_{\lambda_{j}}$. This time, the boundedness of the  sequence $(u_{j})$ follows by the assumption $(V4)$. Moreover, we prove that $(u_{j})$ is a bounded Palais-Smale sequence of $\I$. To show that the bounded sequence $(u_{j})$ converges to a nontrivial weak solution of (\ref{P}), we show that $c_{1}$ is strictly less than the least energy level $m^{\infty}$ of the functional $\I^{\infty}$ associated to the "problem at infinity"
$$
(-\Delta)^{s} u+V_{\infty}u=  f(u) \mbox{ in } \R^{N}.
$$
Together with an accurate description of the sequence as a sum of translated critical points, this allows us to infer that $u_{j}\rightharpoonup u$ in $H^{s}(\R^{N})$, for some nontrivial critical point $u$ of $\I$.

\smallskip

\noindent
Let us recall that when $f$ is an odd function satisfying $(f1)$-$(f4)$, and $V$ is constant, the existence of a radial positive ground state to (\ref{P}) has been proved in \cite{A1} (see also \cite{AFS}) via a minimization argument and by working in the space of radial functions $H^{s}_{rad}(\R^{N})$, which is compactly embedded into $L^{p}(\R^{N})$ for all $p\in (2, 2^{*}_{s})$. 
Here, we present a different proof of this result (see Theorem \ref{thm1}) which is based on the global compactness lemma, which will be also useful to prove Theorem \ref{thm2}. 
In fact, we think that the global compactness lemma is not only interesting for the aim of this paper, but it can be also used to deal with other problems similar to (\ref{P}).
We also point out that by using the methods developed here, we are able to study (\ref{P}) dealing with radial and non-radial potentials in a unified approach.

\smallskip

\noindent
The plan of the paper is the following:  In section $2$ we collect some technical results which will be useful along the paper.
In section $3$ we use the monotonicity trick to prove Theorem \ref{thm1}. In section $4$ we give the proof of Theorem \ref{thm2}.

\section{Preliminaries and functional setting}

\noindent
In this section we give a few results that we are later going to use for the proofs of the main results. \\
For any $s\in (0,1)$ we define $\mathcal{D}^{s, 2}(\R^{N})$ as the completion of $C^{\infty}_{0}(\R^{N})$ with respect to
$$
[u]_{H^{s}(\R^{N})}^{2}=\iint_{\R^{2N}} \frac{|u(x)-u(y)|^{2}}{|x-y|^{N+2s}} \, dx \, dy =|(-\Delta)^{\frac{s}{2}} u|^{2}_{L^{2}(\R^{N})},
$$
that is
$$
\mathcal{D}^{s, 2}(\R^{N})=\left\{u\in L^{2^{*}_{s}}(\R^{N}): [u]_{H^{s}(\R^{N})}<\infty\right\}.
$$
Now, let us introduce the fractional Sobolev space
$$
H^{s}(\R^{N})= \left\{u\in L^{2}(\R^{N}) : \frac{|u(x)-u(y)|}{|x-y|^{\frac{N+2s}{2}}} \in L^{2}(\R^{2N}) \right \}
$$
endowed with the natural norm 
$$
\|u\|_{H^{s}(\R^{N})} = \sqrt{[u]_{H^{s}(\R^{N})}^{2} + |u|_{L^{2}(\R^{N})}^{2}}.
$$

\noindent
For the convenience of the reader we recall from \cite{DPV} the following:
\begin{thm}\label{Sembedding}
Let $s\in (0,1)$ and $N>2s$. Then there exists a sharp constant $S_{*}=S(N, s)>0$, whose exact value can be found in \cite{CT}, 
such that for any $u\in H^{s}(\R^{N})$
\begin{equation}\label{FSI}
|u|^{2}_{L^{2^{*}_{s}}(\R^{N})} \leq S^{-1}_{*} [u]^{2}_{H^{s}(\R^{N})}. 
\end{equation}
Moreover $H^{s}(\R^{N})$ is continuously embedded in $L^{q}(\R^{N})$ for any $q\in [2, 2^{*}_{s}]$ and compactly in $L^{q}_{loc}(\R^{N})$ for any $q\in [2, 2^{*}_{s})$. 
\end{thm}

\begin{remark}
The exact value of the best constant $S_{*}$ appearing in (\ref{FSI}), has been calculated explicitly in \cite{CT}. Moreover, the authors proved that the equality in (\ref{FSI}) holds if and only if 
$$
u(x)=c(\mu^{2}+(x-x_{0})^{2})^{-\frac{N-2s}{2}}
$$
where $c\in \R$, $\mu>0$ and $x_{0}\in \R^{N}$ are fixed constants.
\end{remark}

\noindent
Now, we give some technical lemmas. The first one is a compactness Lions-type lemma whose proof can be found in \cite{Secchi1}. 
\begin{lem}\cite{Secchi1}\label{lions lemma}
Let $N>2s$ and $r\in [2, 2^{*}_{s})$. If $(u_{n})$ is a bounded sequence in $H^{s}(\R^{N})$ and if
$$
\lim_{n \rightarrow \infty} \sup_{y\in \R^{N}} \int_{B_{R}(y)} |u_{n}|^{r} dx=0
$$
where $R>0$,
then $u_{n}\rightarrow 0$ in $L^{t}(\R^{N})$ for all $t\in (2, 2^{*}_{s})$.
\end{lem}

\noindent
Next, we prove the following useful result:
\begin{lem}\label{willem}
If $u_{n}\rightharpoonup u$ in $\mathcal{D}^{s, 2}(\R^{N})$ and $u\in L^{\infty}_{loc}(\R^{N})$, then
\begin{align*}
|u_{n}|^{2^{*}_{s}-2}u_{n} - |u_{n}-u|^{2^{*}_{s}-2} (u_{n}-u) \rightarrow  |u|^{2^{*}_{s}-2}u \quad \mbox{ in } (\mathcal{D}^{s, 2}(\R^{N}))'.
\end{align*} 
\end{lem}
\begin{proof}
Let $\Psi(u)=|u|^{2^{*}_{s}-2}u$ and take $w\in C^{\infty}_{c}(\R^{N})$.
By the mean value theorem, we can see that
$$
|\Psi(u_{n})-\Psi(u_{n}-u)|\leq (2^{*}_{s}-1)[|u_{n}|+|u|]^{2^{*}_{s}-2}|u| \mbox{ a.e. in } \R^{N}.
$$
Then, the H\"older inequality and the Sobolev embedding yield
\begin{align*}
\left|\int_{|x|>R} [\Psi(u_{n})-\Psi(u_{n}-u)]w \, dx\right|\leq C_{1}[|u_{n}|^{2^{*}_{s}-2}_{L^{2^{*}_{s}}(\R^{N})}+|u|^{2^{*}_{s}-2}_{L^{2^{*}_{s}}(\R^{N})}] |w|_{L^{2^{*}_{s}}(\R^{N})} \left(\int_{|x|>R} |u|^{2^{*}_{s}} \, dx\right)^{\frac{1}{2^{*}_{s}}}
\end{align*}
and
\begin{align*}
\left|\int_{|x|>R} \Psi(u)w\, dx\right|\leq |w|_{L^{2^{*}_{s}}(\R^{N})} \left(\int_{|x|>R} |u|^{2^{*}_{s}} \, dx\right)^{\frac{2^{*}_{s}-1}{2^{*}_{s}}}\leq C [w]_{H^{s}(\R^{N})} \left(\int_{|x|>R} |u|^{2^{*}_{s}} \, dx\right)^{\frac{2^{*}_{s}-1}{2^{*}_{s}}}. 
\end{align*}
As a consequence, for any $\varepsilon>0$ there exists $R>0$ such that
\begin{equation}\label{willem1}
\left|\int_{|x|>R} [\Psi(u_{n})-\Psi(u_{n}-u)-\Psi(u)]w\, dx\right|\leq \varepsilon [w]_{H^{s}(\R^{N})}.
\end{equation}
Let us define $M=\sup_{B_{R}}|u|$, so that on $B_{R}$ we get
\begin{equation}\label{Nem}
|\Psi(u_{n})-\Psi(u_{n}-u)|\leq (2^{*}_{s}-1)[|u_{n}|+M]^{2^{*}_{s}-2}M.
\end{equation}
Now, fix $\beta> 1$ such that 
$$
\beta\in \Bigl(\max\left\{\frac{2N}{N+2s}, \frac{N-2s}{4s}\right\}, \frac{N}{2s}\Bigr)
$$
and we set $\alpha:=(2^{*}_{s}-2)\beta=\frac{4s}{N-2s}\beta$.
Then, (\ref{Nem}) becomes
\begin{equation}\label{Nemio}
|\Psi(u_{n})-\Psi(u_{n}-u)|\leq (2^{*}_{s}-1)[|u_{n}|+M]^{\frac{\alpha}{\beta}}M.
\end{equation}
We note that
$$
\alpha> 1 \mbox{ since } \beta> \frac{N-2s}{4s},
$$
and
$$
\alpha<2^{*}_{s} \mbox{ since } \beta< \frac{N}{2s}.
$$
Thus, it follows from the compact embedding $H^{s}_{loc}(\R^{N})\subset L^{\alpha}_{loc}(\R^{N})$ and the properties of Nemytskii operators \cite{W} that 
$$
\Psi(u_{n})-\Psi(u_{n}-u)\rightarrow \Psi(u) \mbox{ in } L^{\beta}(B_{R}).
$$
Hence, by using the H\"older inequality and the Sobolev embedding we can see that 
\begin{align}\label{willem2}
\left|\int_{B_{R}} [\Psi(u_{n})-\Psi(u_{n}-u)-\Psi(u)]w \, dx\right| &\leq \left(\int_{B_{R}} |w|^{\frac{\beta}{\beta-1}} \, dx \right)^{\frac{\beta-1}{\beta}} \left(\int_{B_{R}}|\Psi(u_{n})-\Psi(u_{n}-u)-\Psi(u)|^{\beta} \, dx\right)^{\frac{1}{\beta}} \nonumber\\
&\leq C [w]_{H^{s}(\R^{N})} \left(\int_{B_{R}}|\Psi(u_{n})-\Psi(u_{n}-u)-\Psi(u)|^{\beta} \, dx\right)^{\frac{1}{\beta}} \rightarrow 0.
\end{align}
where in the last inequality we have used the fact that $\frac{\beta}{\beta-1}< 2^{*}_{s}$ because of $\beta> \frac{2N}{N+2s}$.\\
Putting together (\ref{willem1}) and (\ref{willem2}) we obtain the assert.
\end{proof}

\noindent
Finally we recall the following well-known results:
\begin{lem}\cite{BL1}\label{strauss}
Let $P$ and $Q:\R\rightarrow\R$ be a continuous functions satisfying
\begin{equation*}
\lim_{t\rightarrow +\infty} \frac{P(t)}{Q(t)}=0,
\end{equation*}
$\{v_{n}\}_{n}$, $v$ and $w$ be measurable functions from $\R^{N}$ to $\R$, with $w$ bounded, such that 
\begin{align*}
&\sup_{n\in \N} \int_{\R^{N}} |Q(v_{n}(x)) w| \, dx <+\infty, \\
&P(v_{n}(x))\rightarrow v(x) \mbox{ a.e. in } \R^{N}. 
\end{align*}
Then $|(P(v_{n})-v)w|_{L^{1}(\mathcal{B})}\rightarrow 0$, for any bounded Borel set $\mathcal{B}$. \\
Moreover, if we have also
\begin{equation*}
\lim_{t\rightarrow 0} \frac{P(t)}{Q(t)}=0,
\end{equation*}
and
\begin{equation*}
\lim_{|x| \rightarrow \infty} \sup_{n\in \N} |v_{n}(x)|=0,
\end{equation*}
then $|(P(v_{n})-v)w|_{L^{1}(\R^{N})}\rightarrow 0$.
\end{lem}

\begin{lem}\cite{BL}\label{Brezis-Lieb}
Let $\Omega$ be an open set in $\R^{N}$ and let $(u_{n})\subset L^{p}(\R^{N})$, with $1\leq p<\infty$.
If
\begin{compactenum}[(i)]
\item $(u_{n})$ is bounded in $L^{p}(\R^{N})$,
\item $u_{n}\rightarrow u$ almost everywhere in $\Omega$,
\end{compactenum}
then
\begin{align*}
\lim_{n\rightarrow \infty} (|u_{n}|_{L^{p}(\Omega)}^{p}-|u_{n}-u|_{L^{p}(\Omega)}^{p})=|u|_{L^{p}(\Omega)}^{p}.
\end{align*}
\end{lem}

\begin{remark}
In order to simplify the notation in what follows, with $|\cdot|_{q}$ we will always denote the norm of $u\in L^{q}(\R^{N})$.
\end{remark}

\section{Ground state solution when the potential $V$ is constant}
\noindent
In this section we provide the proof of Theorem \ref{thm1}. 
Since we look for positive solution of (\ref{P}), we can suppose that $f(t)=0$ for $t\leq 0$. For simplicity, we also take $K=1$ in $(f3)$.\\
Let $H=\{u\in H^{s}(\R^{N}): \int_{\R^{N}} V(x) u^{2} dx<\infty \}$ endowed with the norm
$$
\|u\|^{2}=\int_{\R^{N}} \left(|(-\Delta)^{\frac{s}{2}}u|^{2}+V(x) u^{2}\right) \,dx.
$$
By using the assumptions $(V2)$ and $(V3)$, it is easy to prove that $\| \cdot \|$ is equivalent to the standard norm in $H^{s}(\R^{N})$.
In order to study weak solutions to (\ref{P}), we look for critical points of the following functional
$$
\I(u)=\frac{1}{2}\|u\|^{2}-\int_{\R^{N}} F(u) dx.
$$
For $\lambda \in [\frac{1}{2}, 1]$, we consider the family of functionals
$$
\I_{\lambda}(u)=\frac{1}{2}\|u\|^{2}-\lambda \int_{\R^{N}} F(u) dx
$$
defined for all $u\in H$.
 
By Theorem \ref{Sembedding} and assumptions on $f$, it is clear that $\I_{\lambda}$ is well defined, $\I_{\lambda}\in C^{1}(H, \R)$ and that its differential is given by  
$$
\langle \I'_{\lambda}(u), \varphi \rangle=\iint_{\R^{2N}} \frac{(u(x)-u(y))}{|x-y|^{N+2s}} (\varphi(x)-\varphi(y)) \, dx dy+\int_{\R^{N}} V(x) u \varphi \, dx-\lambda \int_{\R^{N}} f(u) \varphi \, dx
$$
for any $u, \varphi \in H$.

\noindent
Now, we prove that $\I_{\lambda}$ satisfies the assumptions of Theorem \ref{thm2.1}.
\begin{lem}\label{lemma2.2}
Assume $(V1)$-$(V2)$ and $(f1)$-$(f4)$.  
Then, for almost every $\lambda\in [\frac{1}{2}, 1]$, there is a sequence $(u_{n})\subset H$ such that
\begin{compactenum}[(i)]
\item $(u_{n})$ is bounded;
\item $\I_{\lambda}(u_{n})\rightarrow c_{\lambda}=\inf_{\gamma \in \Gamma} \max_{t\in [0, 1]} \I_{\lambda}(\gamma(t))$
where 
$
\Gamma=\{\gamma\in C([0, 1], H): \gamma(0)=0, \gamma(1)=v_{2}\}
$
for some $v_{2}\in H\setminus\{0\}$ such that $\I_{\lambda}(v_{2})<0$ for all $\lambda \in [\frac{1}{2}, 1]$;
\item $\I'_{\lambda}(u_{n})\rightarrow 0$ on $H^{-1}$.
\end{compactenum}
Moreover, if $V(x)\in L^{\infty}(\R^{N})$, then
\begin{equation}\label{2.3}
c_{\lambda}<\frac{s}{N}\frac{S_{*}^{\frac{N}{2s}}}{\lambda^{\frac{N-2s}{2s}}}.
\end{equation}
\end{lem}
\begin{proof}
We aim to apply Theorem \ref{thm2.1} with $X=H$, $J=[\frac{1}{2}, 1]$, $A(u)=\frac{1}{2} \|u\|^{2}$, and $B(u)=\int_{\R^{N}} F(u) \,dx$.
Clearly, $A(u)\rightarrow \infty$ as $\|u\|\rightarrow \infty$, and by the assumption $(f4)$, it follows that $B(u)\geq 0$ for any $u\in H$.
Now, by using $(f1)$-$(f3)$, we know that for any $\varepsilon >0$ there exists $C_{\varepsilon}>0$ such that
$$
|F(t)|\leq \varepsilon t^{2}+C_{\varepsilon} |t|^{2^{*}_{2}} \mbox{ for all } t\in \R.
$$
Then, by using Theorem \ref{Sembedding}, $(V2)$ and $\lambda\in [\frac{1}{2}, 1]$, we get
\begin{align*}
\I_{\lambda}(u)&\geq \frac{1}{2} \|u\|^{2}-\lambda [\varepsilon |u|_{2}^{2}+C_{\varepsilon} |u|_{2^{*}_{2}}^{2^{*}_{2}}] \\
&\geq \frac{1}{2} \|u\|^{2}- \frac{\varepsilon}{V_{0}} \|u\|^{2}-C_{\varepsilon} S_{*}^{-\frac{2^{*}_{s}}{2}}\|u\|^{2^{*}_{2}}
\end{align*}
so there exist $\alpha>0$ and $r>0$ independent of $\lambda$, such that 
$$
\I_{\lambda}(u)\geq \alpha>0 \mbox{ for any }\|u\|=r.
$$
By using $(f4)$ and $\lambda\in [\frac{1}{2}, 1]$, we can note that
\begin{equation}\label{tend}
\I_{\lambda}(u)\leq \frac{1}{2} \|u\|^{2}-\frac{N-2s}{4N} |u^{+}|_{2^{*}_{s}}^{2^{*}_{s}}-\frac{D}{2q}|u^{+}|_{q}^{q}
\end{equation}
so, taking $\varphi\in H$ such that $\varphi\geq 0$ and $\varphi\neq 0$, we can see that $\I_{\lambda}(tu)\rightarrow -\infty$ as $t\rightarrow \infty$. Hence, there exists $t_{0}>0$ such that $\|t_{0}\varphi\|>r$ and $\I_{\lambda}(t_{0}\varphi)<0$ for all $\lambda \in [\frac{1}{2}, 1]$. Since $\I_{\lambda}(0)=0$, we set $v_{1}=0$ and $v_{2}=t_{0} \varphi$. Therefore, $\I_{\lambda}$ satisfies the assumptions of Theorem \ref{thm2.1}, and we can find a bounded Palais-Smale sequence for $\I_{\lambda}$ at the level $c_{\lambda}$. \\
Finally, we prove the estimate in (\ref{2.3}).
Let $\eta \in C^{\infty}_{0}(\R^{N})$ be a cut-off function such that $0\leq \eta \leq 1$, $\eta=1$ on $B_{r}$ and $\eta=0$ on $\R^{N}\setminus B_{2r}$, where $B_{r}$ denotes the ball in $\R^{N}$ of center at origin and radius $r$.
For $\varepsilon>0$, let us define $u_{\varepsilon}(x)=\eta(x)U_{\varepsilon}(x)$, where
$$
U_{\varepsilon}(x)=\frac{\kappa \varepsilon^{\frac{N-2s}{2}}}{(\varepsilon^{2}+|x|^{2})^{\frac{N-2s}{2}}}
$$
is a solution to
$$
(-\Delta)^{s}u=S_{*}|u|^{2^{*}_{s}-2}u \mbox{ in } \R^{N}
$$
and $\kappa$ is a suitable positive constant depending only on $N$ and $s$.\\
Now we set
$$
v_{\varepsilon}=\frac{u_{\varepsilon}}{|u_{\varepsilon}|_{2^{*}_{s}}}.
$$
As proved in \cite{DoMS, SV}, $v_{\varepsilon}$  satisfies the following useful estimates:
\begin{align}\label{BN1}
[v_{\varepsilon}]^{2}_{H^{s}(\R^{N})}\leq S_{*}+O(\varepsilon^{N-2s}),
\end{align}
\begin{equation}\label{BN2}
|v_{\varepsilon}|_{2}^{2}=
\left\{
\begin{array}{ll}
O(\varepsilon^{2s})  &\mbox{ if } N>4s \\
O(\varepsilon^{2s} |\log(\varepsilon)|) &\mbox{ if } N=4s \\
O(\varepsilon^{N-2s}) &\mbox{ if } N<4s,
\end{array}
\right.
\end{equation}
and
\begin{equation}\label{BN3}
|v_{\varepsilon}|_{q}^{q}=
\left\{
\begin{array}{ll}
O(\varepsilon^{\frac{2N-(N-2s)q}{2}})  &\mbox{ if } q>\frac{N}{N-2s} \\
O(\varepsilon^{\frac{(N-2s)q}{2}}) &\mbox{ if } q<\frac{N}{N-2s}.
\end{array}
\right.
\end{equation}
From the definition of $c_{\lambda}$, we know that 
\begin{equation}\label{cl}
c_{\lambda}\leq \sup_{t\geq 0}\I_{\lambda}(t v_{\varepsilon}).
\end{equation}
Now, we consider the following function for $t\geq 0$
$$
k(t)=\frac{t^{2}}{2}\|v_{\varepsilon}\|^{2}-\frac{t^{2^{*}_{s}}}{2^{*}_{s}}\lambda.
$$ 
We observe that $k(t)$ attains its maximum at $t_{0}=(\lambda^{-1} \|v_{\varepsilon}\|^{2})^{\frac{1}{2^{*}_{s}-2}}$ and 
\begin{equation}\label{mylove}
k(t_{0})=\frac{s}{N}\frac{1}{\lambda^{\frac{N-2s}{2s}}}\|v_{\varepsilon}\|^{\frac{N}{s}}=\frac{s}{N}\frac{1}{\lambda^{\frac{N-2s}{2s}}}\left([v_{\varepsilon}]_{H^{s}(\R^{N})}^{2}+\int_{\R^{N}} V(x) v_{\varepsilon}^{2}\, dx \right)^{\frac{N}{2s}}.
\end{equation}
Let us note that there exists $\tau\in(0, 1)$ such that for $\varepsilon<1$
\begin{equation}\label{MYLOVE}
\sup_{t\in [0, \tau]}\I_{\lambda}(t v_{\varepsilon})\leq \sup_{t\in [0, \tau]} \frac{t^{2}}{2} \|v_{\varepsilon}\|^{2}<\frac{s}{N}\frac{S_{*}^{\frac{N}{2s}}}{\lambda^{\frac{N-2s}{2s}}}.
\end{equation}
On the other hand, in view of $(f4)$, $\lambda\in [\frac{1}{2}, 1]$, \eqref{BN1} and \eqref{mylove},  we get
\begin{align*}
\sup_{t\geq \tau}\I_{\lambda}(t v_{\varepsilon})&\leq \sup_{t\geq 0} k(t)- \lambda\frac{D}{q} \tau^{q} |v_{\varepsilon}|_{q}^{q} \\
&\leq \frac{s}{N}\frac{1}{\lambda^{\frac{N-2s}{2s}}}\left([v_{\varepsilon}]_{H^{s}(\R^{N})}^{2}+\int_{\R^{N}} V(x) v_{\varepsilon}^{2}\, dx \right)^{\frac{N}{2s}}-\frac{D}{2q} \tau^{q} |v_{\varepsilon}|_{q}^{q} \\
&\leq \frac{s}{N} \frac{1}{\lambda^{\frac{N-2s}{2s}}} \left( S_{*} +O(\varepsilon^{N-2s})+\int_{\R^{N}} V(x) v_{\varepsilon}^{2}\, dx\right)^{\frac{N}{2s}}-C_{0}|v_{\varepsilon}|_{q}^{q}. 
\end{align*}
By using the elementary inequality $(a+b)^{p}\leq a^{p}+p(a+b)^{p-1}b$ for all $a, b>0$ and $p\geq 1$, and $V(x)\in L^{\infty}(\R^{N})$, we have
\begin{align*}
\sup_{t\geq \tau}\I_{\lambda}(t v_{\varepsilon})&\leq \frac{s}{N}\frac{S_{*}^{\frac{N}{2s}}}{\lambda^{\frac{N-2s}{2s}}}+O(\varepsilon^{N-2s})+C_{1}\int_{\R^{N}} V(x) |v_{\varepsilon}|^{2}\, dx-C_{0}|v_{\varepsilon}|_{q}^{q} \\
&\leq \frac{s}{N}\frac{S_{*}^{\frac{N}{2s}}}{\lambda^{\frac{N-2s}{2s}}}+O(\varepsilon^{N-2s})+C_{2}|v_{\varepsilon}|_{2}^{2}-C_{0}|v_{\varepsilon}|_{q}^{q}.
\end{align*}
Now, we distinguish the following cases: \\
If $N>4s$, then $q\in (2, 2^{*}_{s})$ and in particular $q>\frac{N}{N-2s}$. Hence, by using (\ref{BN2}) and (\ref{BN3}), we can see that 
\begin{align*}
\sup_{t\geq \tau}\I_{\lambda}(t v_{\varepsilon})&\leq \frac{s}{N}\frac{S_{*}^{\frac{N}{2s}}}{\lambda^{\frac{N-2s}{2s}}}+O(\varepsilon^{N-2s})+O(\varepsilon^{2s})-O(\varepsilon^{\frac{2N-(N-2s)q}{2}}).
\end{align*}
Taking into account $\frac{2N-(N-2s)q}{2}<2s<N-2s$, there exists $\varepsilon_{0}>0$ such that for any $\varepsilon\in (0, \varepsilon_{0})$
\begin{equation}\label{MBRS1}
\sup_{t\geq \tau}\I_{\lambda}(t v_{\varepsilon})<\frac{s}{N}\frac{S_{*}^{\frac{N}{2s}}}{\lambda^{\frac{N-2s}{2s}}}.
\end{equation}
When $N=4s$, then $q\in (2, 4)$ and in particular $q>\frac{N}{N-2s}=2$, so from (\ref{BN2}) and (\ref{BN3}) we deduce that
\begin{align*}
\sup_{t\geq \tau}\I_{\lambda}(t v_{\varepsilon})&\leq \frac{s}{N}\frac{S_{*}^{\frac{N}{2s}}}{\lambda^{\frac{N-2s}{2s}}}+O(\varepsilon^{2s})+O(\varepsilon^{2s}|\log(\varepsilon)|)-O(\varepsilon^{4s-sq}).
\end{align*}
Since $\lim_{\varepsilon\rightarrow 0} \frac{\varepsilon^{4s-sq}}{\varepsilon^{2s}(1+|\log(\varepsilon)|)}=\infty$, for any $\varepsilon$ sufficiently small we have
\begin{equation}\label{MBRS2}
\sup_{t\geq \tau}\I_{\lambda}(t v_{\varepsilon})<\frac{s}{N}\frac{S_{*}^{\frac{N}{2s}}}{\lambda^{\frac{N-2s}{2s}}}.
\end{equation}
Finally, if $2s<N<4s$, then $q\in (\frac{4s}{N-2s}, 2^{*}_{s})$ and in particular $q>\frac{N}{N-2s}$. Hence, observing that $\frac{2N-(N-2s)q}{2}<N-2s$, we get
\begin{align}\label{MBRS3}
\sup_{t\geq \tau}\I_{\lambda}(t v_{\varepsilon})\leq \frac{s}{N}\frac{S_{*}^{\frac{N}{2s}}}{\lambda^{\frac{N-2s}{2s}}}+O(\varepsilon^{N-2s})+O(\varepsilon^{N-2s})-O(\varepsilon^{\frac{2N-(N-2s)q}{2}}) 
<\frac{s}{N}\frac{S_{*}^{\frac{N}{2s}}}{\lambda^{\frac{N-2s}{2s}}}
\end{align}
for any $\varepsilon>0$ small enough. 
Putting together \eqref{cl}, \eqref{MYLOVE} and \eqref{MBRS1}-\eqref{MBRS3}, we can conclude that \eqref{2.3} holds.

\end{proof}

\begin{remark}\label{remark2.3}
Let us note that for $\lambda \in [\frac{1}{2}, 1]$, if $(u_{n})\subset H$ is such that 
$$
\|u_{n}\|\leq C, \quad \I_{\lambda}(u_{n})\rightarrow c_{\lambda}, \quad \I'_{\lambda}(u_{n})\rightarrow 0,
$$
then $u_{n}\geq 0$ in $H$.
In fact, by using $\langle \I_{\lambda}'(u_{n}), u_{n}^{-}\rangle=o_{n}(1)$, where $u^{-}=\min\{u, 0\}$, and the fact that $f(t)=0$ if $t\leq 0$, we can infer that
$$
\int_{\R^{N}} (-\Delta)^{\frac{s}{2}}u_{n} (-\Delta)^{\frac{s}{2}} u_{n}^{-}+ V(x)(u_{n}^{-})^{2} \,dx= o_{n}(1).
$$
On the other hand, we know that
\begin{align*}
\int_{\R^{N}} (-\Delta)^{\frac{s}{2}}u_{n} (-\Delta)^{\frac{s}{2}} u_{n}^{-} \, dx&=\iint_{\R^{2N}} \frac{(u_{n}(x)-u_{n}(y))(u_{n}^{-}(x)-u_{n}^{-}(y))}{|x-y|^{N+2s}} \, dx dy \\
&\geq \iint_{\R^{2N}} \frac{|u_{n}^{-}(x)-u_{n}^{-}(y)|^{2}}{|x-y|^{N+2s}} \, dx dy=[u_{n}^{-}]_{H^{s}(\R^{N})}^{2}.
\end{align*}
Then, $\|u_{n}^{-}\|=o_{n}(1)$, and this allows us to deduce that $\|u_{n}^{+}\|\leq C$, $\I_{\lambda}(u^{+}_{n})\rightarrow c_{\lambda}$ and $\I'_{\lambda}(u^{+}_{n})\rightarrow 0$ as $n\rightarrow \infty$.
\end{remark}

\noindent
Arguing as in \cite{A1, CW, Secchi2} we can prove the following fractional Pohozaev identity:
\begin{lem}\label{lemma2.4}
For $\lambda \in [\frac{1}{2}, 1]$, if $u_{\lambda}$ is a critical point of $\I_{\lambda}$, then $u_{\lambda}$ satisfies the following
Pohozaev identity
\begin{equation}\label{Pohid}
\frac{N-2s}{2}[u_{\lambda}]_{H^{s}(\R^{N})}^{2}+\frac{1}{2}\int_{\R^{N}} \nabla V(x) \cdot x u^{2}_{\lambda} dx=N\int_{\R^{N}} \left[\lambda F(u)-\frac{1}{2} V(x) u^{2}_{\lambda}\right] dx.
\end{equation}
\end{lem}

\begin{remark}\label{remark2.5}
It is easy to check that if $(V1)$-$(V2)$ and $(f1)$-$(f3)$ hold, then there exists $\beta>0$ independent of $\lambda\in [\frac{1}{2}, 1]$ such that any nontrivial critical point $u_{\lambda}$ of $\I_{\lambda}$ verifies $\|u_{\lambda}\|\geq \beta>0$.\\
In fact, by using $(f1)$-$(f3)$, we can see that for any $\varepsilon>0$ there exists $C_{\varepsilon}>0$ such that 
$$
|F(t)|\leq \varepsilon t^{2}+C_{\varepsilon} |t|^{2^{*}_{s}} \mbox{ for all } t\in \R.
$$
Taking into account $\langle \I'_{\lambda}(u_{\lambda}), u_{\lambda}\rangle=0$, $\lambda\leq 1$, $F\geq 0$, the Sobolev embedding, $(V1)$-$(V2)$, we have
\begin{align*}
\|u_{\lambda}\|^{2}=\lambda \int_{\R^{N}} F(u_{\lambda}) \,dx\leq \varepsilon C_{1} \|u_{\lambda}\|^{2}+C_{\varepsilon} C_{2}\|u_{\lambda}\|^{2^{*}_{s}}
\end{align*}
where $C_{1}, C_{2}>0$ depending only on $V_{0}$ and the best constant $S_{*}$. Choosing $\varepsilon>0$ sufficiently small and by using $u_{\lambda}\neq 0$, we deduce that there exists $\beta>0$ such that $\|u_{\lambda}\|\geq \beta>0$.
\end{remark}

\noindent
Now, we establish the following compactness lemma which will be useful to prove Theorem \ref{thm1}.
\begin{lem}\label{lemma2.6}
Assume that $V(x)\equiv V$ and $f$ satisfies $(f1)$-$(f4)$. For $\lambda \in [\frac{1}{2}, 1]$, let $(u_{n})\subset H$ be a bounded sequence  in $H$ such that $u_{n}\geq 0$, $\I_{\lambda}(u_{n})\rightarrow c_{\lambda}$, $\I'_{\lambda}(u_{n})\rightarrow 0$. Moreover, $c_{\lambda}<\frac{s}{N} \frac{S_{*}^{\frac{N}{2s}}}{\lambda^{\frac{N-2s}{2s}}}$. \\
Then there exists a subsequence of $(u_{n})$, which we denote again by $(u_{n})$, and an integer $k\in \N\cup \{0\}$ and $w_{\lambda}^{j}\in H$ for $1\leq j\leq k$ such that 
\begin{compactenum}[(i)]
\item $u_{n}\rightarrow u_{\lambda}$ in $H$ and $\I_{\lambda}'(u_{\lambda})=0$;
\item $w_{\lambda}^{j}\neq 0$ and $\I_{\lambda}'(w_{\lambda}^{j})=0$ for $1\leq j\leq k$;
\item $c_{\lambda}=\I_{\lambda}(u_{\lambda})+\sum_{j=1}^{k} \I_{\lambda}(w_{\lambda}^{j})$
\end{compactenum}
where we agree that in the case $k=0$, the above holds without $w_{\lambda}^{j}$.
\end{lem}
\begin{proof}
We divide the proof in several steps. \\

\begin{compactenum}
\item [{\it Step $1$}] Extracting a subsequence if necessary, we can assume that $u_{n}\rightharpoonup u_{\lambda}$ in $H$ with $u_{\lambda}$ critical point of $\I_{\lambda}$.\\
\end{compactenum}

\noindent
Since $(u_{n})$ is bounded in $H$ and $H$ is a reflexive Banach space, up to a subsequence, we can suppose that $u_{n}\rightharpoonup u_{\lambda}$ in $H$, and, in view of Theorem \ref{Sembedding}, $u_{n}\rightarrow u_{\lambda}$ in $L^{r}_{loc}(\R^{N})$ for all $r\in [1, 2^{*}_{s})$. 
Then, for any $\varphi\in C^{\infty}_{0}(\R^{N})$ we have
\begin{align}\label{criticalpoint}
\langle I'_{\lambda}(u_{n}), \varphi \rangle-\langle I'_{\lambda}(u_{\lambda}), \varphi\rangle =& \langle u_{n}-u_{\lambda}, \varphi \rangle-\lambda \int_{\R^{N}} [g(u_{n})-g(u)]\varphi \, dx  \nonumber \\
&-\lambda \int_{\R^{N}} (|u_{n}|^{2^{*}_{s}-2}u_{n}-|u_{\lambda}|^{2^{*}_{s}-2}u_{\lambda})\varphi \, dx,
\end{align}
where $g(t)=f(t)-(t^{+})^{\frac{N+2s}{N-2s}}$.
Since $u_{n}\rightharpoonup u_{\lambda}$ in $H$, we get
\begin{equation}\label{lim1}
\langle u_{n}-u_{\lambda}, \varphi \rangle\rightarrow 0.
\end{equation}
Moreover, $\{|u_{n}|^{2^{*}_{s}-2}u_{n}-|u_{\lambda}|^{2^{*}_{s}-2}u_{\lambda}\}$ is bounded in $L^{\frac{2^{*}_{s}}{2^{*}_{s}-1}}(\R^{N})$ and 
$$
|u_{n}|^{2^{*}_{s}-2}u_{n}\rightarrow |u_{\lambda}|^{2^{*}_{s}-2}u_{\lambda} \mbox{ a.e. in } \R^{N},
$$
so we obtain that
\begin{equation}\label{lim2}
\int_{\R^{N}} (|u_{n}|^{2^{*}_{s}-2}u_{n}-|u_{\lambda}|^{2^{*}_{s}-2}u_{\lambda})\varphi \, dx \rightarrow 0.
\end{equation}
By using Lemma \ref{strauss}, $(f2)$ and $(f3)$ we can infer that
\begin{equation}\label{lim3}
\int_{\R^{N}} [g(u_{n})-g(u)]\varphi \, dx\rightarrow 0.
\end{equation}
Putting together (\ref{criticalpoint}), (\ref{lim1}), (\ref{lim2}), (\ref{lim3}) and $\I'_{\lambda}(u_{n})\rightarrow 0$, we can see that $\langle\I'_{\lambda}(u_{\lambda}), \varphi\rangle=0$ for any $\varphi\in C^{\infty}_{0}(\R^{N})$. By using the density of $C^{\infty}_{0}(\R^{N})$ in $H^{s}(\R^{N})$, we deduce that $\I'_{\lambda}(u_{\lambda})=0$, that is $(i)$ is satisfied.
Now, we set $v_{n}^{1}=u_{n}-u_{\lambda}$.\\

\begin{compactenum}
\item [{\it Step $2$}] If $\lim_{n\rightarrow \infty}\sup_{z\in \R^{N}} \int_{B_{1}(z)} |v_{n}^{1}|^{2} dx=0$, then $u_{n}\rightarrow u_{\lambda}$ in $H$ and Lemma \ref{lemma2.6} holds with $k=0$.\\
\end{compactenum}

\noindent
From Lemma \ref{lions lemma}, we get
\begin{equation}\label{2.13}
v_{n}^{1}\rightarrow 0 \mbox{ in } L^{t}(\R^{N}),\quad \forall t\in (2, 2^{*}_{s}).  
\end{equation}  
Now, we observe that 
\begin{align*}
\langle \I'_{\lambda}(u_{n}), v_{n}^{1} \rangle &=\langle u_{n}, v_{n}^{1} \rangle-\lambda \int_{\R^{N}} f(u_{n}) v_{n}^{1} \, dx \\
&=\|v_{n}^{1}\|^{2}+\langle u_{\lambda}, v_{n}^{1} \rangle-\lambda \int_{\R^{N}} f(u_{n}) v_{n}^{1} \, dx,
\end{align*}
that is
\begin{align*}
\|v_{n}^{1}\|^{2}=\langle \I'_{\lambda}(u_{n}), v_{n}^{1} \rangle-\langle u_{\lambda}, v_{n}^{1} \rangle+\lambda \int_{\R^{N}} f(u_{n}) v_{n}^{1} \, dx.
\end{align*}
By using $\langle \I'_{\lambda}(u_{\lambda}), v_{n}^{1} \rangle=0$, $\langle \I'_{\lambda}(u_{n}), v_{n}^{1} \rangle=o(1)$ and the definition of $g$, we have
\begin{align*}
\|v_{n}^{1}\|^{2}&=\langle \I'_{\lambda}(u_{n}), v_{n}^{1} \rangle+\lambda \int_{\R^{N}} (f(u_{n})-f(u_{\lambda}))v_{n}^{1} dx \\
&=\lambda \int_{\R^{N}} (g(u_{n})-g(u_{\lambda}))v_{n}^{1} dx+\lambda \int_{\R^{N}} (|u_{n}|^{2^{*}_{s}-2}u_{n}- |u_{\lambda}|^{2^{*}_{s}-2}u_{\lambda})v_{n}^{1} dx+o(1).
\end{align*}
Now, by $(f1)$-$(f3)$, we know that for any $\varepsilon>0$ there exists $C_{\varepsilon}>0$ such that 
\begin{equation}\label{2.14}
|g(t)|\leq \varepsilon (|t|+|t|^{2^{*}_{s}-1})+C_{\varepsilon} |t|^{q-1} \mbox{ for all } t\in \R.
\end{equation}
Therefore, taking into account (\ref{2.13}) and (\ref{2.14}), we obtain
$$
\|v_{n}^{1}\|^{2}= \lambda \int_{\R^{N}} (|u_{n}|^{2^{*}_{s}-2} u_{n} - |u_{\lambda}|^{2^{*}_{s}-2}u_{\lambda})v_{n}^{1}\, dx +o(1).
$$
Since $u_{\lambda}$ is a weak solution to (\ref{P}) and $f$ satisfies $(f1)$-$(f3)$, we can argue as in \cite{A1, CW} to infer that $u_{\lambda}\in L^{\infty}(\R^{N})$. 
Then, by using Lemma \ref{willem}, we  can see that 
\begin{align}\label{2.15}
\left| \int_{\R^{N}} [|u_{n}|^{2^{*}_{s}-2}u_{n} - |u_{\lambda}|^{2^{*}_{s}-2}u_{\lambda} - |u_{n}-u_{\lambda}|^{2^{*}_{s}-2} (u_{n}-u_{\lambda})]\varphi \, dx \right| = o(1) \|\varphi\|, \quad \forall \varphi \in H.
\end{align}  
Taking $\varphi=v_{n}^{1}=u_{n}- u_{\lambda}$ in (\ref{2.15}), we deduce
\begin{equation}\label{2.16}
\|v_{n}^{1}\|^{2} = \lambda \int_{\R^{N}} |v_{n}^{1}|^{2^{*}_{s}} \, dx + o(1).
\end{equation}
Let us note that Lemma \ref{Brezis-Lieb} yields
\begin{equation}\label{2.17}
\|v_{n}^{1}\|^{2}= \|u_{n}\|^{2}-\|u_{\lambda}\|^{2}+o(1)
\end{equation}
and
\begin{equation}\label{2.18}
\int_{\R^{N}} |v_{n}^{1}|^{2^{*}_{s}} \, dx = \int_{\R^{N}} |u_{n}|^{2^{*}_{s}} \, dx - \int_{\R^{N}} |u_{\lambda}|^{2^{*}_{s}}\, dx +o(1). 
\end{equation}
At this point, we aim to prove that 
\begin{equation}\label{2.19}
\int_{\R^{N}} G(v_{n}^{1})\, dx = \int_{\R^{N}} G(u_{n})\, dx - \int_{\R^{N}} G(u_{\lambda})\, dx +o(1). 
\end{equation}
By (\ref{2.14}) and the mean value theorem, it holds
\begin{equation*}
|G(u_{n}) - G(v_{n}^{1})|\leq C[(|v_{n}^{1}| + |u_{\lambda}|) + (|v_{n}^{1}| + |u_{\lambda}|)^{2^{*}_{s}-1}]|u_{\lambda}|. 
\end{equation*}
Fixed $R>0$, by using the H\"older's inequality we obtain
\begin{align}\begin{split}\label{2.20}
&\int_{\{|x|\geq R\}} |G(u_{n}) - G(v_{n}^{1})|\, dx \\
&\leq C \left(\int_{\{|x|\geq R\}} |v_{n}^{1}|^{2}\, dx \right)^{\frac{1}{2}} \left( \int_{\{|x|\geq R\}} |u_{\lambda}|^{2}\, dx \right)^{\frac{1}{2}} + C  \int_{\{|x|\geq R\}} |u_{\lambda}|^{2}\, dx\\
&+ C \left(\int_{\{|x|\geq R\}} |v_{n}^{1}|^{2^{*}_{s}}\, dx \right)^{\frac{2^{*}_{s}-1}{2^{*}_{s}}} \left( \int_{\{|x|\geq R\}} |u_{\lambda}|^{2^{*}_{s}}\, dx \right)^{\frac{1}{2^{*}_{s}}} + C  \int_{\{|x|\geq R\}} |u_{\lambda}|^{2^{*}_{s}}\, dx. 
\end{split}\end{align}
On the other hand, by $(f1)$-$(f3)$, we get
\begin{equation}\label{2.21}
\int_{\{|x|\geq R\}} |G(u_{\lambda})|\, dx \leq C \int_{\{|x|\geq R\}} |u_{\lambda}|^{2}\, dx + C \int_{\{|x|\geq R\}} |u_{\lambda}|^{2^{*}}\, dx.  
\end{equation}
Combining (\ref{2.20}) with (\ref{2.21}), we deduce that for any $\varepsilon>0$ there exists $R>0$ such that
\begin{equation}\label{2.22}
\int_{\{|x|\geq R\}} |G(u_{n}) - G(v_{n}^{1}) - G(u_{\lambda})|\, dx \leq \varepsilon. 
\end{equation}
We recall that $\lim_{t\rightarrow \infty}\frac{G(t)}{|t|^{2^{*}_{s}}}=0$ by $(f3)$, and $|u_{n}|_{2^{*}_{s}}\leq C$ for all $n\in \N$, being $(u_{n})$ bounded in $H$. Then, by using Lemma \ref{strauss}, we can see that
\begin{equation}\label{2.23}
\lim_{n\rightarrow \infty}\int_{\{|x|\leq R\}} |G(u_{n}) - G(u_{\lambda})|\, dx =0, 
\end{equation}
and, in similar way, 
\begin{equation}\label{2.24}
\lim_{n\rightarrow \infty}\int_{\{|x|\leq R\}} |G(v_{n}^{1}) |\, dx =0. 
\end{equation}
Hence, (\ref{2.22}), (\ref{2.23}) and (\ref{2.24}) show that (\ref{2.19}) is verified. Putting together (\ref{2.17})-(\ref{2.19}), we have
\begin{align}\label{2.25}
c_{\lambda} - \I_{\lambda}(u_{\lambda})&=\I_{\lambda}(u_{n})- \I_{\lambda}(u_{\lambda})+o(1)  \nonumber \\
&=\Bigl[\frac{1}{2} \|u_{n}\|^{2} - \lambda \int_{\R^{N}} G(u_{n})\, dx-\frac{\lambda}{2^{*}_{s}} |u_{n}|_{2^{*}_{s}}^{2^{*}_{s}}\Bigr]-\Bigl[\frac{1}{2} \|u_{\lambda}\|^{2} - \lambda \int_{\R^{N}} G(u_{\lambda})\, dx-\frac{\lambda}{2^{*}_{s}} |u_{\lambda}|_{2^{*}_{s}}^{2^{*}_{s}}\Bigr]+o(1) \nonumber \\
&=\frac{1}{2} \|v_{n}^{1}\|^{2} - \lambda \int_{\R^{N}} G(v_{n}^{1})\, dx - \frac{\lambda}{2^{*}_{s}} |v_{n}^{1}|_{2^{*}_{s}}^{2^{*}_{s}} +o(1)
\end{align} 
By using (\ref{2.13}), (\ref{2.14}) and (\ref{2.25}) we can infer that
\begin{equation}\label{2.26}
c_{\lambda} - \I_{\lambda}(u_{\lambda})= \frac{1}{2} \|v_{n}^{1}\|^{2} -\frac{\lambda}{2^{*}_{s}} \int_{\R^{N}}|v_{n}^{1}|^{2^{*}_{s}}\, dx +o(1). 
\end{equation} 
Since $\I'_{\lambda}(u_{\lambda})=0$, from Lemma \ref{lemma2.4} it follows that $\I_{\lambda}(u_{\lambda})=\frac{s}{N}[u_{\lambda}]_{H^{s}(\R^{N})}^{2}\geq 0$. Then, in view of (\ref{2.26}), we get 
$$
c_{\lambda}- \I_{\lambda}(u_{\lambda})<\frac{s}{N} \frac{S_{*}^{\frac{N}{2s}}}{\lambda^{\frac{N-2s}{2s}}}.
$$
Now, we may assume that $\|v_{n}^{1}\|^{2}\rightarrow L\geq 0$. By (\ref{2.16}), it follows that $\lambda |v_{n}^{1}|_{2^{*}_{s}}^{2^{*}_{s}}\rightarrow L$. \\
Let us suppose that $L>0$. Then, by using the Sobolev embedding we know that
$$
|v_{n}^{1}|^{2}_{2^{*}_{s}}S_{*}\leq \|v_{n}^{1}\|^{2},
$$
so we can deduce that $L\geq \frac{S_{*}^{\frac{N}{2s}}}{\lambda^{\frac{N-2s}{2s}}}$.\\
This fact and (\ref{2.26}) yield  
$$
c_{\lambda} - \I_{\lambda}(u_{\lambda})=\frac{s}{N}L\geq \frac{s}{N} \frac{S_{*}^{\frac{N}{2s}}}{\lambda^{\frac{N-2s}{2s}}}
$$
which gives a contradiction. Hence, $\|v_{n}^{1}\|\rightarrow 0$ as $n\rightarrow \infty$.   \\

\begin{compactenum}
\item [{\it Step $3$}] If there exists $(z_{n})\subset \R^{N}$ such that $\int_{B_{1}(z_{n})} |v_{n}^{1}|^{2}\, dx \rightarrow d>0$, then, up to a subsequence, the following conditions hold 
\begin{compactenum}[(1)]
\item $|z_{n}|\rightarrow \infty$;
\item $u_{n}(\cdot + z_{n})\rightharpoonup w_{\lambda}\neq 0$ in $H$;
\item $\I'_{\lambda}(w_{\lambda})=0$. \\
\end{compactenum}
\end{compactenum}

\noindent
We may assume that there exists $(z_{n})\subset \R^{N}$ such that 
$$
\int_{B_{1}(z_{n})} |v_{n}^{1}|^{2}\, dx \geq \frac{d}{2}>0.
$$
Set $\tilde{v}^{1}_{n}(x)= v^{1}_{n}(x+z_{n})$. Then $\tilde{v}^{1}_{n}$ is bounded in $H$ and we may suppose that $\tilde{v}_{n}^{1}\rightharpoonup \tilde{v}^{1}$ in $H$.\\
Since 
$$
\int_{B_{1}(0)} |\tilde{v}_{n}^{1}|^{2}\, dx \geq \frac{d}{2}
$$
we get
$$
\int_{B_{1}(0)} |\tilde{v}^{1}|^{2}\, dx \geq \frac{d}{2}
$$
that is $\tilde{v}^{1}\neq 0$. From the fact that $v_{n}^{1}\rightharpoonup 0$ in $H$, we deduce that $(z_{n})$ is unbounded, so we may assume that $|z_{n}|\rightarrow \infty$.
Now, we set $\tilde{u}_{n}(x)= u_{n}(x+z_{n})\rightharpoonup w_{\lambda}\neq 0$. As in Step 1, we can see that $\langle \I'_{\lambda}(\tilde{u}_{n}), \varphi \rangle-\langle \I'_{\lambda}(w_{\lambda}), \varphi \rangle\rightarrow 0$, for all $\varphi\in C^{\infty}_{0}(\R^{N})$.
On the other hand, being $|z_{n}|\rightarrow \infty$, we have for all $\varphi\in C^{\infty}_{0}(\R^{N})$
$$
\langle \I'_{\lambda}(\tilde{u}_{n}), \varphi \rangle=\langle \I'_{\lambda}(u_{n}), \varphi(\cdot-z_{n}) \rangle\rightarrow 0, 
$$
so we can conclude that $\langle \I'_{\lambda}(w_{\lambda}), \varphi \rangle =0$, for all $\varphi\in C^{\infty}_{0}(\R^{N})$.\\

\begin{compactenum}
\item [{\it Step $4$}] If there exists $m\geq 1$, $(y_{n}^{k})\subset \R^{N}$, $w_{\lambda}^{k}\in H$ for $1\leq k\leq m$ such that 
\begin{compactenum}[(i)]
\item $|y_{n}^{k}|\rightarrow \infty$, $|y_{n}^{k} - y_{n}^{h}|\rightarrow \infty$ if $k\neq h$, 
\item $u_{n}(\cdot + y_{n}^{k})\rightharpoonup w_{\lambda}^{k}\neq 0$ in $H$, for any $1\leq k\leq m$, 
\item $w_{\lambda}^{k}\geq 0$ and $\I'_{\lambda}(w_{\lambda}^{k})=0$ for any $1\leq k\leq m$,
\end{compactenum}
then one of the following conclusions must hold:
\begin{compactenum}[(1)]
\item If $\sup_{z\in \R^{N}} \int_{B_{1}(z)} |u_{n} - u_{0} - \sum_{k=1}^{m} w_{\lambda}^{k}(\cdot-y_{n}^{k})|^{2} \, dx\rightarrow 0$, then 
\begin{equation*}
\left\|u_{n} - u_{0} - \sum_{k=1}^{m} w_{\lambda}^{k}(\cdot-y_{n}^{k})\right\|\rightarrow 0. 
\end{equation*}

\item If there exists $(z_{n})\subset \R^{N}$ such that 
\begin{equation*}
\int_{B_{1}(z_{n})} \left|u_{n} - u_{0} - \sum_{k=1}^{m} w_{\lambda}^{k}(\cdot-y_{n}^{k})\right|^{2} \, dx \rightarrow d>0,
\end{equation*}
then up to a subsequence, the following conditions hold
\begin{compactenum}[(i)]
\item $|z_{n}|\rightarrow \infty$, $|z_{n} - y_{n}^{k}|\rightarrow \infty$ for any $1\leq k\leq m$,
\item $u_{n}(\cdot + z_{n})\rightharpoonup w_{\lambda}^{m+1}\neq 0$ in $H$, 
\item $w_{\lambda}^{m+1}\geq 0$ and $\I'_{\lambda}(w_{\lambda}^{m+1})=0$. \\
\end{compactenum}
\end{compactenum}
\end{compactenum}

\noindent
Assume that $(1)$ holds. Set $\xi_{n}=u_{n} - u_{0} - \sum_{k=1}^{m} w_{\lambda}^{k}(\cdot-y_{n}^{k})$. Then, by using Lemma \ref{lions lemma} we can see that 
\begin{equation}\label{2.27}
\xi_{n}\rightarrow 0 \, \mbox{ in } \, L^{t}(\R^{N}) \mbox{ for all } t\in (2, 2^{*}_{s}).
\end{equation}
By using the definition of $\xi_{n}$ and the fact that $\langle \I'_{\lambda}(u_{\lambda}), \xi_{n}\rangle=0=\langle \I'_{\lambda}(w^{k}_{\lambda}), \xi_{n}\rangle$, we can infer
\begin{align*}
\|\xi_{n}\|^{2}=\langle \I'_{\lambda}(u_{n}), \xi_{n}\rangle+\lambda \int_{\R^{N}} (f(u_{n})-f(u_{\lambda})) \xi_{n} \, dx-\lambda \sum_{k=1}^{m} \int_{\R^{N}} f(w^{k}_{\lambda}) \xi_{n}(\cdot+y_{n}^{k}) \, dx.
\end{align*}
In view of (\ref{2.14}) and (\ref{2.27}), we deduce that
\begin{equation*}
\|\xi_{n}\|^{2}= \lambda \int_{\R^{N}} (|u_{n}|^{2^{*}_{s}-2} u_{n} - |u_{\lambda}|^{2^{*}_{s}-2} u_{\lambda}) \xi_{n} \, dx - \lambda \sum_{k=1}^{m} \int_{\R^{N}} |w_{\lambda}^{k}|^{2^{*}_{s}-2} w_{\lambda}^{k} \xi_{n}(\cdot + y_{n}^{k}) \, dx +o(1).
\end{equation*}
Recalling (\ref{2.15}), we can observe that
\begin{align*}
\|\xi_{n}\|^{2} &= \lambda \int_{\R^{N}} |u_{n}-u_{\lambda}|^{2^{*}_{s}-2} (u_{n}- u_{\lambda}) \xi_{n} \, dx - \lambda \int_{\R^{N}} |w_{\lambda}^{1}|^{2^{*}_{s}-2} w_{\lambda}^{1} \xi_{n}(\cdot + y_{n}^{1}) \, dx\\
&- \lambda \sum_{k=2}^{m} \int_{\R^{N}} |w_{\lambda}^{k}|^{2^{*}_{s}-2} w_{\lambda}^{k} \xi_{n}(\cdot + y_{n}^{k}) \, dx +o(1) \\
&= \lambda \int_{\R^{N}} |u_{n}(\cdot + y_{n}^{1})-u_{\lambda}(\cdot + y_{n}^{1})|^{2^{*}_{s}-2} (u_{n}(\cdot + y_{n}^{1})- u_{\lambda}(\cdot + y_{n}^{1})) \xi_{n}(\cdot + y_{n}^{1}) \, dx \\
&- \lambda \int_{\R^{N}} |w_{\lambda}^{1}|^{2^{*}_{s}-2} w_{\lambda}^{1} \xi_{n}(\cdot + y_{n}^{1}) \, dx- \lambda \sum_{k=2}^{m} \int_{\R^{N}} |w_{\lambda}^{k}|^{2^{*}_{s}-2} w_{\lambda}^{k} \xi_{n}(\cdot + y_{n}^{k}) \, dx +o(1). 
\end{align*}
Since $|y_{n}^{1}|\rightarrow \infty$, and $u_{n}(\cdot + y_{n}^{1})\rightharpoonup w_{\lambda}^{1}$ in $H$, we deduce that $u_{n}(\cdot + y_{n}^{1})-u_{\lambda}(\cdot+y_{n}^{1})\rightharpoonup w_{\lambda}^{1}$ in $H$. \\
As a consequence we have
\begin{align*}
\|\xi_{n}\|^{2} &= \lambda \int_{\R^{N}} |u_{n}(\cdot + y_{n}^{1})-u_{\lambda}(\cdot + y_{n}^{1}) - w_{\lambda}^{1}|^{2^{*}_{s}-2}(u_{n}(\cdot + y_{n}^{1})-u_{\lambda}(\cdot + y_{n}^{1}) - w_{\lambda}^{1}) \xi_{n}(\cdot + y_{n}^{1}) \, dx \\
&- \lambda \sum_{k=2}^{m} \int_{\R^{N}} |w_{\lambda}^{k}|^{2^{*}_{s}-2} w_{\lambda}^{k} \xi_{n}(\cdot + y_{n}^{k}) \, dx +o(1). 
\end{align*}
Iterating this procedure, we obtain that
\begin{equation}\label{2.28}
\|\xi_{n}\|^{2} = \lambda \int_{\R^{N}}|\xi_{n}|^{2^{*}_{s}} \, dx +o(1). 
\end{equation}
Now, since $u_{n}(\cdot + y_{n}^{1})- u_{\lambda}(\cdot + y_{n}^{1})\rightharpoonup w_{\lambda}^{1}$ in $H$, we can argue as in Step 2  to see  that
\begin{align*}
c_{\lambda} - \I_{\lambda}(u_{\lambda})& = \frac{1}{2} \|u_{n}- u_{\lambda}\|^{2}- \lambda \int_{\R^{N}} G(u_{n}- u_{\lambda})\, dx - \frac{\lambda}{2^{*}_{s}}\int_{\R^{N}} |u_{n}- u_{\lambda}|^{2^{*}_{s}}\, dx + o(1)\\
&= \frac{1}{2} \|u_{n}(\cdot + y_{n}^{1})- u_{\lambda}(\cdot + y_{n}^{1}) - w_{\lambda}^{1}\|^{2}- \lambda \int_{\R^{N}} G(u_{n}(\cdot + y_{n}^{1})- u_{\lambda}(\cdot + y_{n}^{1}) - w_{\lambda}^{1})\, dx \\
&- \frac{\lambda}{2^{*}_{s}}\int_{\R^{N}} |u_{n}(\cdot + y_{n}^{1})- u_{\lambda}(\cdot + y_{n}^{1})- w_{\lambda}^{1}|^{2^{*}_{s}}\, dx + \I_{\lambda}(w_{\lambda}^{1}) + o(1). 
\end{align*}
Continuing this process, we obtain that
\begin{equation}\label{2.29}
c_{\lambda} - \I_{\lambda}(u_{\lambda})- \sum_{k=1}^{m} \I_{\lambda}(w_{\lambda}^{k})= \frac{1}{2}\|\xi_{n}\|^{2} - \lambda \int_{\R^{N}} G(\xi_{n})\, dx - \frac{\lambda}{2^{*}_{s}} \int_{\R^{N}} |\xi_{n}|^{2^{*}_{s}}\, dx + o(1), 
\end{equation}
which together with (\ref{2.27}), yields
\begin{equation}\label{2.30}
c_{\lambda} - \I_{\lambda}(u_{\lambda})- \sum_{k=1}^{m} \I_{\lambda}(w_{\lambda}^{k})= \frac{1}{2}\|\xi_{n}\|^{2} - \frac{\lambda}{2^{*}_{s}} \int_{\R^{N}} |\xi_{n}|^{2^{*}_{s}}\, dx + o(1). 
\end{equation}
Then, taking into account (\ref{2.28}) and (\ref{2.30}), we can argue as in Step 2 to infer that 
$$
\left\|u_{n} - u_{0} - \sum_{k=1}^{m} w_{\lambda}^{k}(\cdot-y_{n}^{k})\right\|=\|\xi_{n}\|\rightarrow 0 \quad \mbox{ as } n\rightarrow \infty.
$$ 
Now, we assume that $(2)$ holds. The proof of this is standard (see \cite{JT}), so we skip the details here.\\

\begin{compactenum}
\item [{\it Step $5$}] Conclusion. \\
\end{compactenum}

\noindent
By using the Step $1$, we can see that Lemma \ref{lemma2.6} (i) holds. If the assumption of Step $2$ is verified, then Lemma \ref{lemma2.6} holds with $k=0$. Otherwise, the assumption of Step $3$ holds. We set $(y_{n}^{1})= (z_{n})$ and $w_{\lambda}^{1}= w_{\lambda}$. Now, if $(1)$ of Step $4$ holds with $m=1$, from (\ref{2.30}), we obtain the conclusion of Lemma \ref{lemma2.6}. If not, $(2)$ of Step $4$ must hold, and by setting $(y_{n}^{2})=(z_{n})$, $w_{\lambda}^{2}=w^{2}_{\lambda}$, we iterate Step $4$. Then, to conclude the proof, we have to show that $(1)$ of Step $4$ must occur after a finite number of iterations.
Let us note that, for all $m\geq 1$ we have
\begin{equation*}
\lim_{n\rightarrow \infty} \left( \|u_{n}\|^{2} - \|u_{\lambda}\|^{2} - \sum_{k=1}^{m} \|w_{\lambda}^{k}\|^{2} \right) = \lim_{n\rightarrow \infty} \left \| u_{n} - u_{\lambda} - \sum_{k=1}^{m} w_{\lambda}^{k} (\cdot - y_{n}^{k}) \right\|^{2}\geq 0.
\end{equation*}  
In fact, by using $(i), (ii)$ of Step 4 and $u_{n}\rightharpoonup u_{\lambda}$ in $H$, we can see that
\begin{align}\label{JT}
\left \| u_{n} - u_{\lambda} - \sum_{k=1}^{m} w_{\lambda}^{k} (\cdot - y_{n}^{k}) \right\|^{2}&=\|u_{n}\|^{2}+\|u_{\lambda}\|^{2}+\sum_{k=1}^{m} \|w_{\lambda}^{k}\|^{2}-2\langle u_{n}, u_{\lambda} \rangle-2\sum_{k=1}^{m} \langle u_{n}, w_{\lambda}^{k} (\cdot - y_{n}^{k})\rangle  \nonumber \\
&+2\sum_{k=1}^{m} \langle u_{\lambda}, w_{\lambda}^{k} (\cdot - y_{n}^{k})\rangle+2\sum_{h, k} \langle w_{\lambda}^{h} (\cdot - y_{n}^{h}), w_{\lambda}^{k} (\cdot - y_{n}^{k})\rangle  \nonumber \\
&=\|u_{n}\|^{2}-\|u_{\lambda}\|^{2}-\sum_{k=1}^{m} \|w_{\lambda}^{k}\|^{2}+o(1)
\end{align}
On the other hand, by Remark \ref{remark2.5} we know that $\|w_{\lambda}^{k}\|\geq \beta$ for some $\beta>0$ independent of $\lambda$. Thus, by using (\ref{JT}) and the fact that $(u_{n})$ is bounded in $H$,  we deduce that 
$(1)$ in Step $4$ must occur after a finite number of iterations. This together with (\ref{2.30}), allow us to infer that Lemma \ref{lemma2.6} holds. 

\end{proof}

\noindent
Before giving the proof of the main result of this section, we prove the following lemma.
\begin{lem}\label{lemma3.1}
Under the same assumptions of Theorem \ref{thm1}, for almost every $\lambda\in [\frac{1}{2}, 1]$, $\I_{\lambda}$ has a positive critical point.
\end{lem}
\begin{proof}
By Lemma \ref{lemma2.2}, for almost every $\lambda\in [\frac{1}{2}, 1]$, there exists a bounded sequence $(u_{n})\subset H$ such that 
\begin{equation}\label{2.111}
\I_{\lambda}(u_{n})\rightarrow c_{\lambda}, \quad \I'_{\lambda}(u_{n})\rightarrow 0.
\end{equation}
By using Remark \ref{remark2.3}, we may assume that $u_{n}\geq 0$ in $H$.
In addition, $c_{\lambda}\in \left(0, \frac{s}{N}\frac{S_{*}^{\frac{N}{2s}}}{\lambda^{\frac{N-2s}{2s}}}\right)$. Then, up to a subsequence, we may suppose that $u_{n}\rightharpoonup u_{\lambda}$ in $H$. If $u_{\lambda}\neq 0$, then we have finished. Otherwise, we may suppose that $u_{n}\rightharpoonup 0$ in $H$.
Now, we aim to show that there exists $\delta>0$ such that 
\begin{equation}\label{3.1}
\lim_{n\rightarrow \infty} \sup_{y\in \R^{N}}\int_{B_{1}(y)} |u_{n}|^{2} \, dx \geq \delta>0.
\end{equation}
If (\ref{3.1}) does not occur, by Lemma \ref{lions lemma}, it follows that 
\begin{equation}\label{3.2}
u_{n}\rightarrow 0 \mbox{ in } L^{t}(\R^{N}), \quad \forall t\in (2, 2^{*}_{s}).
\end{equation}
By using (\ref{2.14}) and (\ref{3.2}), we obtain that $\int_{\R^{N}} G(u_{n})\, dx=o(1)$ and $\int_{\R^{N}} g(u_{n})u_{n} \,dx=o(1)$.
This and (\ref{2.111}) yield
\begin{equation}\label{3.3}
\frac{1}{2}\|u_{n}\|^{2}-\frac{\lambda}{2^{*}_{s}}\int_{\R^{N}} u_{n}^{2^{*}_{s}} dx=c_{\lambda}+o(1)
\end{equation}
and
\begin{equation}\label{3.4}
\|u_{n}\|^{2}-\lambda \int_{\R^{N}} u_{n}^{2^{*}_{s}} dx=o(1).
\end{equation}
Since $c_{\lambda}>0$, we may assume that $\|u_{n}\|^{2}\rightarrow L$ for some $L>0$. By using the Sobolev embedding, we can infer that $L\geq \frac{S_{*}^{\frac{N}{2s}}}{\lambda^{\frac{N-2s}{2s}}}$. \\
This together with (\ref{3.3}) and (\ref{3.4}), imply $c_{\lambda}\geq \frac{s}{N}\frac{S_{*}^{\frac{N}{2s}}}{\lambda^{\frac{N-2s}{2s}}}$, which is a contradiction. Then, (\ref{3.1}) holds, and we can find $(y_{n})\subset \R^{N}$ such that $|y_{n}|\rightarrow \infty$ and $\int_{B_{1}(y_{n})} |u_{n}|^{2} \,dx\geq \frac{\delta}{2}>0$. \\
Set $v_{n}=u_{n}(\cdot+y_{n})$. By using (\ref{2.111}) we derive that $\I_{\lambda}(v_{n})\rightarrow c_{\lambda}$ and $\I'_{\lambda}(v_{n})\rightarrow 0$. In view of (\ref{3.1}), we can deduce that $v_{n} \rightharpoonup v_{\lambda}\neq 0$ in $H$ and $\I'_{\lambda}(v_{\lambda})=0$. It is easy to check that $v_{\lambda}\geq 0$ in $\R^{N}$, and due to the fact that $v_{\lambda}\neq 0$, we get $v_{\lambda}>0$ in $\R^{N}$. In fact, if there exists $x_{0}\in \R^{N}$ such that 
$v_{\lambda}(x_{0})=0$, then we can see that 
$$
(-\Delta)^{s}v_{\lambda}(x_{0})=(-\Delta)^{s}v_{\lambda}(x_{0})+V v_{\lambda}(x_{0})=\lambda f(v_{\lambda}(x_{0}))=0.
$$
By using the representation formula for the fractional Laplacian \cite{DPV}, we have
$$
\int_{\R^{N}} \frac{v_{\lambda}(x_{0}+y)+v_{\lambda}(x_{0}-y)}{|x_{0}-y|^{N+2s}} \, dy=0
$$
which gives $v_{\lambda}=0$, that is a contradiction.
\end{proof}

\noindent
Now, we are ready to prove the existence of positive ground state to (\ref{P}) when $V$ is  constant.
\begin{proof}[Proof of Theorem \ref{thm1}]
By using Lemma \ref{lemma3.1}, for almost every $\lambda\in [\frac{1}{2}, 1]$, there exists $(u_{n})\subset H$ such that $u_{n}\geq 0$ in $H$, $\I_{\lambda}(u_{n})\rightarrow c_{\lambda}\in \Bigl(0, \frac{s}{N}\frac{S_{*}^{\frac{N}{2s}}}{\lambda^{\frac{N-2s}{2s}}}\Bigr)$, $\I'_{\lambda}(u_{n})\rightarrow 0$ and $u_{n}\rightharpoonup u_{\lambda}> 0$ in $H$. \\
In view of Lemma \ref{lemma2.6}, we can see that 
$$
c_{\lambda}=\I_{\lambda}(u_{\lambda})+\sum_{j=1}^{k} \I_{\lambda}(w_{\lambda}^{j}),
$$
$\I'_{\lambda}(u_{\lambda})=0$ and $\I'_{\lambda}(w^{j}_{\lambda})=0$ for $1\leq j\leq k$. 
By Lemma \ref{lemma2.4}, we deduce that $\I_{\lambda}(u_{\lambda})> 0$ and $ \I_{\lambda}(w_{\lambda}^{j})\geq 0$ for $1\leq j\leq k$, so we have $c_{\lambda}\geq \I_{\lambda}(u_{\lambda})> 0$. Thus, there exists $(\lambda_{n})\subset [\frac{1}{2}, 1]$ such that $\lambda_{n}\rightarrow 1$, $u_{\lambda_{n}}\in H$, $u_{\lambda_{n}}>0$, $\I'_{\lambda_{n}}(u_{\lambda_{n}})=0$, $c_{\lambda_{n}}\geq \I_{\lambda_{n}}(u_{\lambda_{n}})>0$ and $c_{\lambda_{n}}\in \left(0, \frac{s}{N}\frac{S_{*}^{\frac{N}{2s}}}{\lambda_{n}^{\frac{N-2s}{2s}}}\right)$.
By using the fact that $\I'_{\lambda_{n}}(u_{\lambda_{n}})=0$ and Lemma \ref{lemma2.4}, we infer that 
$$
c_{\lambda_{n}}\geq \I_{\lambda_{n}}(u_{\lambda_{n}})=\frac{s}{N} [u_{\lambda_{n}}]_{H^{s}(\R^{N})}^{2}>0.
$$
Moreover, in view of the Sobolev embedding, we have $|u_{\lambda_{n}}|_{2^{*}_{s}}\leq C$ for all $n\in \N$.
Putting together $(f1)$-$(f3)$ and  Lemma \ref{lemma2.4}, we can see that for any $\varepsilon>0$ there exists $C_{\varepsilon}>0$ such that 
\begin{align*}
\frac{N-2s}{2N}[u_{\lambda_{n}}]^{2}+\frac{1}{2}\int_{\R^{N}} V u_{\lambda_{n}}^{2} \, dx=\lambda_{n} \int_{\R^{N}} F(u_{\lambda_{n}}) \, dx\leq \varepsilon |u_{\lambda_{n}}|_{2}^{2}+C_{\varepsilon}|u_{\lambda_{n}}|_{2^{*}_{s}}^{2^{*}_{s}}.
\end{align*}
which implies that
\begin{align*}
\frac{V}{2}\int_{\R^{N}} u_{\lambda_{n}}^{2} \, dx\leq \varepsilon |u_{\lambda_{n}}|_{2}^{2}+C_{\varepsilon}C.
\end{align*}
Therefore, choosing $\varepsilon\in (0, \frac{V}{2})$, we can deduce that $(u_{\lambda_{n}})$ is bounded in $H$. Now, we can assume that there exists $\lim_{n\rightarrow \infty} \I_{\lambda_{n}}(u_{\lambda_{n}})$. Since the map $\lambda \mapsto c_{\lambda}$ is continuous from the left (see Theorem \ref{thm2.1}), we have
$$
0\leq \lim_{n\rightarrow \infty} \I_{\lambda_{n}}(u_{\lambda_{n}})\leq c_{1}<\frac{s}{N} S_{*}^{\frac{N}{2s}}.
$$ 
Then, by using the fact that
$$
\I(u_{\lambda_{n}})=\I_{\lambda_{n}}(u_{\lambda_{n}})+(\lambda_{n}-1)\int_{\R^{N}} F(u_{\lambda_{n}}) \, dx
$$
and $\|u_{\lambda_{n}}\|\leq C$, we can infer that
\begin{equation}\label{3.5}
0\leq \lim_{n\rightarrow \infty} \I(u_{\lambda_{n}})\leq c_{1}<\frac{s}{N} S_{*}^{\frac{N}{2s}}
\end{equation}
and
\begin{equation}
\lim_{n\rightarrow \infty} \I'(u_{\lambda_{n}})=0.
\end{equation}
In view of Remark \ref{remark2.5}, there exists $\beta>0$ independent of $\lambda_{n}$ such that $\|u_{\lambda_{n}}\|\geq \beta$.
Moreover, we know that $(u_{\lambda_{n}})$ is bounded in $H$, so we can use similar arguments to the proof of Lemma \ref{lemma3.1} to obtain the existence of a positive solution $u_{0}$ to (\ref{P}).\\
By using Lemma \ref{lemma2.6}, we can also see that
$$
\I(u_{0})\leq \lim_{n\rightarrow \infty} \I(u_{\lambda_{n}})\leq c_{1}<\frac{s}{N} S_{*}^{\frac{N}{2s}}.
$$
Let us define
$$
m=\inf\{\I(u): u\in H, u\neq 0, \I'(u)=0\}.
$$
Since $\I'(u_{0})=0$, we get $m\leq \I(u_{0})<\frac{s}{N} S_{*}^{\frac{N}{2s}}$, and by using Lemma \ref{lemma2.4}, we get $0\leq m<\frac{s}{N} S_{*}^{\frac{N}{2s}}$.
From the definition of $m$, we can find $(u_{n})\subset H$ such that $\I(u_{n})\rightarrow m$ and $\I'(u_{n})=0$.
Taking into account Remark \ref{remark2.5}, we deduce that $\|u_{n}\|\geq \beta>0$ for some $\beta$ independent of $n$. Moreover, it is easy to see that $(u_{n})$ is bounded in $H$. \\
In virtue of Remark \ref{remark2.3}, we may assume that $u_{n}\geq 0$ in $H$. Then, taking in mind that $\|u_{n}\|\geq \beta>0$, we can proceed as in the proof of Lemma \ref{lemma3.1}, to show that there exists $(v_{n})\subset H$ such that $v_{n}\geq 0$ in $H$, $v_{n}\rightharpoonup v_{0}>0$ in $H$, $\I(v_{n})\rightarrow m$ and $\I'(v_{n})=0$. By using Lemma \ref{lemma2.6}, we can infer that $\I(v_{0})\leq m$ and $\I'(v_{0})=0$.
Since $\I'(v_{0})=0$, we also have $\I(v_{0})\geq m$. Then, we have proved that $v_{0}>0$ is such that $\I(v_{0})= m$ and $\I'(v_{0})=0$.

\end{proof}

\section{Ground state solution when $V$ is not constant}

\noindent
In this last section we provide the proof of the existence of ground state to (\ref{P}) under the assumptions that $V$ is a non constant potential. For this reason, we will assume that $V(x)\not\equiv V_{\infty}$.\\
For $\lambda\in [\frac{1}{2}, 1]$, we introduce the following family of functionals defined for $u\in H$
\begin{equation}
\I_{\lambda}^{\infty}(u)=\frac{1}{2}\left[\int_{\R^{N}} \left(|(-\Delta)^{\frac{s}{2}}u|^{2}+V_{\infty} u^{2}\right) \,dx\right]-\lambda \int_{\R^{N}} F(u) \, dx.
\end{equation}
Following \cite{A1}, we can prove the following result:
\begin{lem}\label{lemma4.1}
For $\lambda\in [\frac{1}{2}, 1]$, if $w_{\lambda}\in H$ is a nontrivial critical point of $\I_{\lambda}^{\infty}$, then there exists $\gamma_{\lambda}\in C([0, 1], H)$ such that $\gamma_{\lambda}(0)=0$, $\I_{\frac{1}{2}}^{\infty}(\gamma_{\lambda}(1))<0$, $w_{\lambda}\in \gamma_{\lambda}([0, 1])$, $0\notin \gamma_{\lambda}((0, 1])$ and $\max_{t\in [0, 1]} \I_{\lambda}^{\infty}(\gamma_{\lambda}(t))=\I_{\lambda}^{\infty}(w_{\lambda})$.
\end{lem}
\begin{proof}
Let
\begin{equation*}
\gamma_{\lambda}(t)(x)= 
\left\{
\begin{array}{ll}
w_{\lambda}(\frac{x}{t}) &\mbox{ for } t>0\\
0 &\mbox{ for } t=0.   
\end{array}
\right.
\end{equation*}
Then we can see  
\begin{align}\label{Igamma}
&\|\gamma_{\lambda}(t)\|^{2}=t^{N-2s}[w_{\lambda}]^{2}_{H^{s}(\R^{N})}+t^{N}V_{\infty} |w_{\lambda}|^{2}_{2} \nonumber \\
&\I^{\infty}_{\lambda}(\gamma_{\lambda}(t))=\frac{t^{N-2s}}{2} [w_{\lambda}]^{2}_{H^{s}(\R^{N})}+\frac{t^{N}}{2}V_{\infty} |w_{\lambda}|^{2}_{2} - t^{N} \lambda \int_{\R^{N}} F(w_{\lambda}) \,dx. 
\end{align}
Hence $\gamma_{\lambda}\in C([0, \infty), H)$. \\
By using Lemma \ref{lemma2.4}, we know that 
\begin{equation}\label{Gomega}
\lambda \int_{\R^{N}}F(w_{\lambda}) \, dx=\frac{N-2s}{2N} [w_{\lambda}]^{2}_{H^{s}(\R^{N})}+\frac{1}{2}V_{\infty}|w_{\lambda}|_{2}^{2}, 
\end{equation}
which together with (\ref{Igamma}) and (\ref{Gomega}), yields
$$
\I^{\infty}_{\lambda}(\gamma_{\lambda}(t))=\left(\frac{t^{N-2s}}{2}-t^{N}\frac{N-2s}{2N}\right)[w_{\lambda}]^{2}_{H^{s}(\R^{N})}. 
$$
Then, after a suitable change of scale, we can obtain the desired path.

\end{proof}

\begin{remark}\label{remark4.2}
From Theorem \ref{thm1}, we know that if $(f1)$-$(f4)$ hold, then for $\lambda\in [\frac{1}{2}, 1]$, $\I_{\lambda}^{\infty}$ has a ground state.
\end{remark}

\begin{lem}\label{lemma4.3}
Under the same assumptions of Theorem \ref{thm2}, we have that for almost every $\lambda\in [\frac{1}{2}, 1]$, $\I_{\lambda}$ has a positive critical point.
\end{lem}
\begin{proof}
From Lemma \ref{lemma2.2} and Remark \ref{remark2.3}, we may assume that for almost every $\lambda\in [\frac{1}{2}, 1]$, there exists $(u_{n})\subset H$ such that $u_{n}\geq 0$ in $H$, $u_{n}\rightharpoonup u_{\lambda}$ in $H$, $\I_{\lambda}(u_{n})\rightarrow c_{\lambda}\in (0, \frac{s}{N}\frac{S_{*}^{\frac{N}{2s}}}{\lambda^{\frac{N-2s}{2s}}})$ and $\I'_{\lambda}(u_{n})\rightarrow 0$. \\
Now, our claim is to prove that $u_{\lambda}\neq 0$. We argue by contradiction, and we suppose that $u_{\lambda}= 0$. As in the proof of Lemma \ref{lemma3.1}, we can find $(y_{n})\subset \R^{N}$ such that $|y_{n}|\rightarrow \infty$ and $v_{n}= u_{n}(\cdot + y_{n})\rightharpoonup v_{\lambda}\neq 0$ in $H$. 
Furthermore, by using the fact that $u_{n}\rightharpoonup 0$ in $H$, we can see that $\I_{\lambda}^{\infty}(u_{n})\rightarrow c_{\lambda}$ and $(\I_{\lambda}^{\infty})'(u_{n})\rightarrow 0$ hold. Thus, $\I_{\lambda}^{\infty}(v_{n}) \rightarrow c_{\lambda}$ and $(\I_{\lambda}^{\infty})'(v_{n})\rightarrow 0$. \\
Since $v_{n}\rightharpoonup v_{\lambda}\neq 0$ in $H$, it holds $(\I_{\lambda}^{\infty})'(v_{\lambda})= 0$.
In view of Lemma \ref{lemma2.6} we get $c_{\lambda}\geq \I_{\lambda}^{\infty}(v_{\lambda})$. From Remark \ref{remark4.2} it follows that $\I_{\lambda}^{\infty}$ has a ground state $w_{\lambda}$. 
Thus, $c_{\lambda}\geq \I_{\lambda}^{\infty}(w_{\lambda})$. By Lemma \ref{lemma4.1}, we can find a path $\gamma_{\lambda}\in C([0,1], H)$ such that $\gamma_{\lambda}(0)=0$, $\I_{\frac{1}{2}}^{\infty}(\gamma_{\lambda}(1))<0$, $w_{\lambda}\in \gamma_{\lambda}([0, 1])$, $0\notin \gamma_{\lambda}((0, 1])$ and $\max_{t\in [0, 1]} \I_{\lambda}^{\infty}(\gamma_{\lambda}(t))=\I_{\lambda}^{\infty}(w_{\lambda})$.\\ 
Therefore we obtain
\begin{equation}\label{zao1}
c_{\lambda}\geq \I_{\lambda}^{\infty}(w_{\lambda})= \max_{t\in [0, 1]} \I_{\lambda}^{\infty}(\gamma_{\lambda}(t)).
\end{equation} 
Taking into account $(V3)$, $V\not\equiv V_{\infty}$ and $0\notin \gamma_{\lambda}((0, 1])$, we can see that $\I_{\lambda}(\gamma_{\lambda}(t))< \I_{\lambda}^{\infty}(\gamma_{\lambda}(t))$ for all $t\in (0, 1]$. Now, we take $v_{1}=0$ and $v_{2}= \gamma_{\lambda}(1)$ in Theorem \ref{thm2.1}. Then, by using the definition of $c_{\lambda}$ and (\ref{zao1}) we get 
\begin{equation}\label{zao2}
c_{\lambda}\leq \max_{t\in [0, 1]} \I_{\lambda}(\gamma_{\lambda}(t)) < \max_{t\in [0,1]} \I_{\lambda}^{\infty}(\gamma_{\lambda}(t))\leq c_{\lambda}, 
\end{equation}
which gives a contradiction. As a consequence $u_{\lambda}\neq 0$, and by applying the maximum principle \cite{CabSir} we can deduce that $u_{\lambda}>0$. 

\end{proof}

\noindent
At this point we establish the following lemma which will be fundamental to prove Theorem \ref{thm2}.
\begin{lem}\label{lemma4.4}
Assume that $(V1)$-$(V3)$ and $(f1)$-$(f5)$ are satisfied. For $\lambda\in [\frac{1}{2}, 1]$, let $(u_{n})\subset H$ be a bounded sequence in $H$ such that $u_{n}\geq 0$ in $H$, $\I_{\lambda}(u_{n})\rightarrow c_{\lambda}\in \Bigl(0, \frac{s}{N}\frac{S_{*}^{\frac{N}{2s}}}{\lambda^{\frac{N-2s}{2s}}}\Bigr)$ and $\I'_{\lambda}(u_{n})\rightarrow 0$. \\
Then there exists a subsequence of $(u_{n})$, which we denote again by $(u_{n})$, such that
\begin{compactenum}[(i)]
\item $u_{n}\rightharpoonup u$ in $H$ and $\I'_{\lambda}(u_{\lambda})=0$,
\item $\I_{\lambda}(u_{\lambda})\leq c_{\lambda}$.
\end{compactenum}
\end{lem}
\begin{proof}
Since $(u_{n})$ is bounded in $H$, up to a subsequence, we may suppose that $u_{n}\rightharpoonup u_{\lambda}$ in $H$. 
Then, proceeding as in the proof of Step $1$ in Lemma \ref{lemma2.6}, and by using $(V3)$, we can see that $\I'_{\lambda}(u_{\lambda})=0$, that is $(i)$ is satisfied.\\
Set $w_{n}^{1}=u_{n}-u_{\lambda}$.
Similarly to the proof of Lemma \ref{lemma2.6}, we can deduce that
\begin{align}\label{4.2}
c_{\lambda}-\I_{\lambda}(u_{\lambda})=\frac{1}{2}\|w_{n}^{1}\|^{2}-\lambda \int_{\R^{N}} G(w_{n}^{1}) \, dx-\frac{\lambda}{2^{*}_{s}} |w_{n}^{1}|_{2^{*}_{s}}^{2^{*}_{s}}+o(1).
\end{align}
At this point, we aim to prove that for any $\varphi \in H$
\begin{equation}\label{4.3}
\left|\int_{\R^{N}} (g(u_{n})-g(u_{\lambda})-g(w_{n}^{1})) \varphi\, dx \right|= o(1) \|\varphi\|. 
\end{equation}
By using $(f5)$ and the mean value theorem, we can see that
\begin{equation*}
|g(u_{n}) - g(w_{n}^{1})|\leq C[1+(|w_{n}^{1}| + |u_{\lambda}| + |v_{n}^{1}|)^{2^{*}_{s}-2}]|u_{\lambda}|. 
\end{equation*}
Fix $R>0$ and by applying the H\"older's inequality we obtain
\begin{align}\begin{split}\label{4.4}
&\int_{\{|x|\geq R\}} |g(u_{n}) - g(w_{n}^{1})||\varphi|\, dx \\
&\leq C \left(\int_{\{|x|\geq R\}} |u_{\lambda}|^{2}\, dx \right)^{\frac{1}{2}} \|\varphi\| +C\left( \int_{\{|x|\geq R\}} |u_{\lambda}|^{2^{*}_{s}}\, dx \right)^{\frac{N+2s}{2N}}  \|\varphi\|\\
&+ C \left(\int_{\{|x|\geq R\}} |w_{n}^{1}|^{2^{*}_{s}}\, dx \right)^{\frac{2s}{N}} \left( \int_{\{|x|\geq R\}} |u_{\lambda}|^{2^{*}_{s}}\, dx \right)^{\frac{1}{2^{*}_{s}}} \|\varphi\| 
\end{split}\end{align}
and
\begin{align}\label{4.5}
\int_{\{|x|\geq R\}} |g(u_{\lambda}) \varphi|\, dx &\leq C\int_{|x|\geq R} |u_{\lambda}||\varphi| \, dx+C \int_{|x|\geq R} |u_{\lambda}|^{2^{*}_{s}-1}|\varphi| \, dx  \nonumber\\
&\leq C \left(\int_{\{|x|\geq R\}} |u_{\lambda}|^{2}\, dx\right)^{\frac{1}{2}} \|\varphi\| + C \left(\int_{\{|x|\geq R\}} |u_{\lambda}|^{2^{*}}\, dx\right)^{\frac{2^{*}_{s}-1}{2^{*}_{s}}} \|\varphi\| .  
\end{align}
Putting together (\ref{4.4}) and (\ref{4.5}), we deduce that for any $\varepsilon>0$ there exists $R>0$ such that
\begin{equation}\label{4.6}
\left|\int_{\{|x|\geq R\}} (g(u_{n}) - g(u_{\lambda})- g(w_{n}^{1})) \varphi\, dx\right| \leq \varepsilon \|\varphi\|. 
\end{equation}
Let us note that
\begin{align*}
\int_{|x|\leq R} |g(u_{n}) - g(u_{\lambda})| |\varphi| \, dx\leq \left(\int_{|x|\leq R} |g(u_{n}) - g(u_{\lambda})|^{\frac{2^{*}_{s}}{2^{*}_{s}-1}} \, dx \right)^{\frac{2^{*}_{s}-1}{2^{*}_{s}}} \left(\int_{|x|\leq R} |\varphi|^{2^{*}_{s}} \, dx \right)^{\frac{1}{2^{*}_{s}}}.
\end{align*}
Recalling that $\lim_{t\rightarrow \infty}\frac{g(t)^{\frac{2^{*}_{s}}{2^{*}_{s}-1}}}{|t|^{2^{*}_{s}}}=0$ and $(u_{n})$ is bounded in $L^{2^{*}_{s}}(\R^{N})$, we can apply Lemma \ref{strauss} to infer that
\begin{equation*}
\lim_{n\rightarrow \infty} \int_{|x|\leq R} |g(u_{n})|^{\frac{2^{*}_{s}}{2^{*}_{s}-1}} \, dx=\int_{|x|\leq R} |g(u_{\lambda})|^{\frac{2^{*}_{s}}{2^{*}_{s}-1}} \, dx.
\end{equation*}
By using the Dominated Convergence Theorem, we have
\begin{equation}\label{4.7}
\int_{\{|x|\leq R\}} |g(u_{n}) - g(u_{\lambda})| |\varphi|\, dx =o(1)\|\varphi\|
\end{equation}
and  
\begin{equation}\label{4.8}
\int_{\{|x|\leq R\}} |g(w_{n}^{1})| |\varphi| \, dx =o(1)\|\varphi\|. 
\end{equation}
Hence, (\ref{4.6})-(\ref{4.8}) show that (\ref{4.3}) is verified. Now, for $\lambda \in [\frac{1}{2}, 1]$, let us introduce the following functionals on $H$
\begin{align*}
&H_{\lambda}(u)= \frac{1}{2} \|u\|^{2} - \lambda \int_{\R^{N}} G(u) \, dx - \frac{\lambda}{2^{*}_{s}} \int_{\R^{N}} |u|^{2^{*}_{s}} \, dx, \\
&H_{\lambda}^{\infty}(u)= \frac{1}{2} \int_{\R^{N}} \left(|(-\Delta)^{\frac{s}{2}} u|^{2} + V_{\infty}u^{2}\right)\, dx - \lambda \int_{\R^{N}} G(u) \, dx - \frac{\lambda}{2^{*}_{s}} \int_{\R^{N}} |u|^{2^{*}_{s}} \, dx,  \\
&J^{\infty}_{\lambda}(u)= \frac{1}{2} \int_{\R^{N}} \left(|(-\Delta)^{\frac{s}{2}} u|^{2} + V_{\infty}u^{2}\right)\, dx- \frac{\lambda}{2^{*}_{s}} \int_{\R^{N}} |u|^{2^{*}_{s}} \, dx. 
\end{align*}
Then, (\ref{4.2}) becomes 
\begin{equation}\label{4.9}
c_{\lambda}- \I_{\lambda}(u_{\lambda})= H_{\lambda}(w_{n}^{1}) + o(1). 
\end{equation}
By using (\ref{2.15}) and (\ref{4.3}) we have, for any $\varphi \in H$, 
\begin{equation}\label{4.10}
|\langle \I_{\lambda}'(u_{n}) - \I_{\lambda}'(u_{\lambda}), \varphi	 \rangle - \langle H_{\lambda}'(w_{n}^{1}), \varphi \rangle |= o(1) \|\varphi\|, 
\end{equation}
which gives 
\begin{equation}\label{4.11}
H_{\lambda}'(w_{n}^{1})=o(1). 
\end{equation}
Taking into account (\ref{4.9}), (\ref{4.11}), $(V3)$ and the fact that $w_{n}^{1}\rightharpoonup 0$ in $H$, we get
\begin{equation}\label{4.12}
c_{\lambda}- \I_{\lambda}(u_{\lambda})= H_{\lambda}^{\infty}(w_{n}^{1})+ o(1)
\end{equation}
and 
\begin{equation}\label{4.13}
(H_{\lambda}^{\infty})'(w_{n}^{1})=o(1). 
\end{equation}
Now we distinguish two cases. \\
$(1)$ $\lim_{n\rightarrow \infty} \sup_{y\in \R^{N}} \int_{B_{1}(y)} |w_{n}^{1}|^{2}\, dx =0$. \\
By using Lemma \ref{lions lemma} we have
\begin{equation}\label{4.14}
w_{n}^{1}\rightarrow 0 \, \mbox{ in } \, L^{t}(\R^{N}), \quad \forall t\in (2, 2^{*}_{s}). 
\end{equation}
Putting together (\ref{2.14}) and (\ref{4.12})-(\ref{4.14}), we can deduce that
\begin{equation*}
c_{\lambda}- \I_{\lambda}(u_{\lambda})= J^{\infty}_{\lambda}(w_{n}^{1})+ o(1) \, \mbox{ and } \, J'_{\lambda}(w_{n}^{1})=o(1) 
\end{equation*}
which gives
\begin{equation*}
c_{\lambda}- \I_{\lambda}(u_{\lambda})= \frac{\lambda s}{N} |w_{n}^{1}|_{2^{*}_{s}}^{2^{*}_{s}} +o(1)
\end{equation*}
and then $c_{\lambda}\geq \I_{\lambda}(u_{\lambda})$. \\
$(2)$ $\lim_{n\rightarrow \infty} \sup_{y\in \R^{N}} \int_{B_{1}(y)} |w_{n}^{1}|^{2}\, dx \geq \delta_{1}$ for some $\delta_{1}>0$.\\
Thus, there exists $y_{n}^{1}\in \R^{N}$, $|y_{n}^{1}|\rightarrow \infty$ such that $\int_{B_{1}(y_{n}^{1})} |w_{n}^{1}|^{2}\, dx\geq \frac{\delta_1}{2}$. As a consequence, we can see that $w_{n}^{1}(\cdot + y_{n}^{1})\rightharpoonup w_{\lambda}^{1}\neq 0$ in $H$,  
\begin{equation}\label{4.15}
c_{\lambda}- \I_{\lambda}(u_{\lambda})= H_{\lambda}^{\infty}(w_{n}^{1}(\cdot+ y_{n}^{1}))+ o(1)
\end{equation}
and 
\begin{equation}\label{4.16}
(H_{\lambda}^{\infty})'(w_{n}^{1}(\cdot+ y_{n}^{1}))=o(1). 
\end{equation}  
By (\ref{4.16}) we have $(H_{\lambda}^{\infty})'(w_{n}^{1})=0$. Now, if $c_{\lambda}- \I_{\lambda}(u_{\lambda})< \frac{s}{N} \frac{S_{*}^{\frac{N}{2s}}}{\lambda^{\frac{N-2s}{2s}}}$, then we can proceed as in the proof of Lemma \ref{lemma2.6} to obtain the thesis. \\
Otherwise, we set $w_{n}^{2}= w_{n}^{1}(\cdot+y_{n}^{1})- w_{\lambda}^{1}$, and repeating the same arguments of (\ref{4.9}) and (\ref{4.11}), we can see 
\begin{equation}\label{4.17}
c_{\lambda}- \I_{\lambda}(u_{\lambda})- H_{\lambda}^{\infty}(w_{\lambda}^{1}) + o(1) = H_{\lambda}^{\infty}(w_{n}^{2})
\end{equation}
and 
\begin{equation}\label{4.18}
(H_{\lambda}^{\infty})'(w_{n}^{2})=o(1). 
\end{equation}  
Then, as before, the following cases can occur
\begin{equation}\label{4.19}
\lim_{n\rightarrow \infty} \sup_{y\in \R^{N}} \int_{B_{1}(y)} |w_{n}^{2}|^{2}\, dx =0
\end{equation}
or 
\begin{equation}\label{4.20}
\lim_{n\rightarrow \infty} \sup_{y\in \R^{N}} \int_{B_{1}(y)} |w_{n}^{2}|^{2}\, dx \geq \delta_{2}>0. 
\end{equation}
Let us suppose that (\ref{4.19}) is true. Then, by the case $(1)$ we deduce that $c_{\lambda}- \I_{\lambda}(u_{\lambda})- H_{\lambda}^{\infty}(w_{\lambda}^{1})\geq 0$, and by Lemma \ref{lemma2.4} we get $H_{\lambda}^{\infty}(w_{\lambda}^{1})\geq 0$. This two facts give $c_{\lambda}- \I_{\lambda}(u_{\lambda})\geq 0$. \\
Now, we can suppose that (\ref{4.20}) holds. Repeating this procedure, we can find $w_{n}^{i} \in H$, $y_{n}^{i}\in \R^{N}$, $|y_{n}^{i}|\rightarrow \infty$, $i\in \N$ such that $w_{n}^{i}(\cdot + y_{n}^{i})\rightharpoonup w_{\lambda}^{i}\neq 0$ in $H$, $(H_{\lambda}^{\infty})'(w_{\lambda}^{i})=0$, 
\begin{equation}\label{4.21}
c_{\lambda}- \I_{\lambda}(u_{\lambda}) - \sum_{i=1}^{j} H_{\lambda}^{\infty}(w_{\lambda}^{i})+o(1)= H_{\lambda}^{\infty}(w_{n}^{j+1})
\end{equation}
and 
\begin{equation}\label{4.22}
(H_{\lambda}^{\infty})'(w_{n}^{j+1})=o(1), 
\end{equation}
where 
\begin{equation*}
w_{n}^{j+1}= w_{n}^{j}(\cdot + y_{n}^{j})- w_{\lambda}^{j}, \quad j\in \N. 
\end{equation*}
Since $(H_{\lambda}^{\infty})'(w_{\lambda}^{i})=0$, we can use Lemma \ref{lemma2.4} to  get
\begin{equation}\label{4.23}
H_{\lambda}^{\infty}(w_{\lambda}^{i})= \frac{s}{N} [w_{\lambda}^{i}]_{H^{s}(\R^{N})}^{2}. 
\end{equation}
Now we show that there exists $\alpha>0$ independent of $i$ such that
\begin{equation}\label{4.24}
[w_{\lambda}^{i}]_{H^{s}(\R^{N})}\geq \alpha. 
\end{equation}
In fact, by using $(H_{\lambda}^{\infty})'(w_{\lambda}^{i})=0$, $\lambda \in [\frac{1}{2}, 1]$, and $(f1)$-$(f3)$, we can see that for any $\varepsilon>0$ there exists $C_{\varepsilon}>0$ such that 
\begin{align*}
\int_{\R^{N}} \left(|(-\Delta)^{\frac{s}{2}}w^{i}_{\lambda}|^{2} + V_{\infty}|w_{\lambda}^{i}|^{2}\right)\, dx &\leq \varepsilon \int_{\R^{N}} |w_{\lambda}^{i}|^{2} \, dx + C_{\varepsilon} \int_{\R^{N}} |w_{\lambda}^{i}|^{2^{*}_{s}}\, dx\\
& \leq \frac{\varepsilon}{V_{\infty}} \int_{\R^{N}} V_{\infty} |w_{\lambda}^{i}|^{2}\, dx +C_{\varepsilon} \int_{\R^{N}} |w_{\lambda}^{i}|^{2^{*}_{s}}\, dx.
\end{align*} 
Choosing $\varepsilon\in (0,V_{\infty})$, we can infer that 
\begin{equation*}
[w_{\lambda}^{i}]_{H^{s}(\R^{N})}^{2} \leq C |w_{\lambda}^{i}|_{2^{*}_{s}}^{2^{*}_{s}}, 
\end{equation*}
which together with the Sobolev inequality, gives (\ref{4.24}).\\
Then, putting together (\ref{4.23}) and (\ref{4.24}), at some $j=k$, we obtain that
\begin{equation*}
c_{\lambda} - \I_{\lambda}(u_{\lambda})- \sum_{i=1}^{j} H_{\lambda}^{\infty}(w_{\lambda}^{i})< \frac{s}{N}\frac{S_{*}^{\frac{N}{2s}}}{\lambda^{\frac{N-2s}{2s}}}.   
\end{equation*}
The conclusion follows by Lemma \ref{lemma2.6}.

\end{proof}

\noindent
We end this section giving the proof of Theorem \ref{thm2}.

\begin{proof}[Proof of Theorem \ref{thm2}]
Taking into account Lemma \ref{lemma4.3}, for almost every $\lambda\in [\frac{1}{2}, 1]$ there exists $(u_{n})\subset H$ such that $\I_{\lambda}(u_{n})\rightarrow c_{\lambda}\in \left(0, \frac{s}{N} \frac{S_{*}^{\frac{N}{2s}}}{\lambda^{\frac{N-2s}{2s}}}\right)$, $\I'_{\lambda}(u_{n})\rightarrow 0$, $u_{n}\rightharpoonup u_{\lambda}\neq 0$ in $H$. \\
By using Lemma \ref{lemma4.4}, we deduce that $\I_{\lambda}(u_{\lambda})\leq c_{\lambda}$ and $\I'_{\lambda}(u_{\lambda})=0$. Hence, we can find a sequence $\lambda_{n}\in [\frac{1}{2}, 1]$ such that $\lambda_{n}\rightarrow 1$, $c_{\lambda_{n}}\in (0, \frac{s}{N} \frac{S_{*}^{\frac{N}{2s}}}{\lambda^{\frac{N-2s}{2s}}})$, $u_{\lambda_{n}}\in H$ such that $\I'_{\lambda_{n}}(u_{\lambda_{n}})=0$, $\I_{\lambda_{n}}(u_{\lambda_{n}})\leq c_{\lambda_{n}}$.
At this point, we show that there exists a positive constant $C$ such that
\begin{equation}\label{bound}
\|u_{\lambda_{n}}\|\leq C \mbox{ for all } n\in \N.
\end{equation}
By using $(V4)$, we know that there exists $\theta\in (0, 2s)$ such that 
\begin{equation}\label{squas}
|\max\{x\cdot \nabla V, 0\}|_{\frac{N}{2s}}\leq \theta S_{*}.
\end{equation} 
Taking into account $\I_{\lambda_{n}}(u_{\lambda_{n}})\leq c_{\frac{1}{2}}$, $\I'_{\lambda_{n}}(u_{\lambda_{n}})=0$, Lemma \ref{lemma2.4}, H\"older inequality, Theorem \ref{Sembedding} and (\ref{squas}), we can infer that
\begin{align}\label{bound1}
s[u_{\lambda_{n}}]_{H^{s}(\R^{N})}^{2}&=\frac{N}{2}[u_{\lambda_{n}}]_{H^{s}(\R^{N})}^{2}+\frac{N}{2} \int_{\R^{N}} V(x) u_{\lambda_{n}}^{2}\, dx+\frac{1}{2} \int_{\R^{N}} x\cdot \nabla V(x) u_{\lambda_{n}}^{2} \, dx-\lambda_{n} N \int_{\R^{N}} F(u_{\lambda_{n}}) \, dx \nonumber\\
&=N \I_{\lambda_{n}}(u_{\lambda_{n}})+\frac{1}{2} \int_{\R^{N}} x\cdot \nabla V(x) u_{\lambda_{n}}^{2} \, dx \leq N c_{\frac{1}{2}}+\frac{\theta}{2}  [u_{\lambda_{n}}]^{2},
\end{align}
which implies that $[u_{\lambda_{n}}]_{H^{s}(\R^{N})}\leq C$ for any $n\in \N$.
Putting together $\I'_{\lambda_{n}}(u_{\lambda_{n}})=0$, $\lambda_{n}\in [\frac{1}{2}, 1]$, $(f1)$-$(f3)$ and the Sobolev inequality, we have for any $\varepsilon>0$  
\begin{align}\label{bound2}
V_{0} |u_{\lambda_{n}}|_{2}^{2}&\leq \int_{\R^{N}} \left(|(-\Delta)^{\frac{s}{2}} u_{\lambda_{n}}|^{2}+V(x)u^{2}_{\lambda_{n}}\right) \, dx \nonumber\\
&\leq \int_{\R^{N}} f(u_{\lambda_{n}})u_{\lambda_{n}} \, dx \nonumber\\
&\leq \varepsilon |u_{\lambda_{n}}|_{2}^{2}+C_{\varepsilon} |u_{\lambda_{n}}|_{2^{*}_{s}}^{2^{*}_{s}} \nonumber\\
&\leq \varepsilon |u_{\lambda_{n}}|_{2}^{2}+C'_{\varepsilon} [u_{\lambda_{n}}]_{H^{s}(\R^{N})}^{2^{*}_{s}}.
\end{align}
Choosing $\varepsilon \in (0, V_{0})$ and by using $[u_{\lambda_{n}}]_{H^{s}(\R^{N})}\leq C$, we can see that (\ref{bound2}) yields $|u_{\lambda_{n}}|_{2}\leq C$ for all $n\in \N$. In view of $(V2)$ and $(V3)$, we deduce that $0\leq \int_{\R^{N}} V(x) u^{2}_{\lambda_{n}} \, dx\leq V_{\infty} |u_{\lambda_{n}}|_{2}^{2}\leq V_{\infty} C^2$, which completes the proof of (\ref{bound}).\\
Now, we can note that 
$$
\I(u_{\lambda_{n}})=\I_{\lambda_{n}}(u_{\lambda_{n}})+(\lambda_{n}-1) \int_{\R^{N}} F(u_{\lambda_{n}}) \, dx,
$$
so we can infer that 
$$
\lim_{n\rightarrow \infty} \I(u_{\lambda_{n}})\leq c_{1}<\frac{s}{N} S_{*}^{\frac{N}{2s}}
$$
and
$$
\lim_{n\rightarrow \infty} \I'(u_{\lambda_{n}})=0.
$$
In view of Remark \ref{remark2.5}, we know that there exists $\beta>0$ independent of $\lambda_{n}$ such that $\|u_{\lambda_{n}}\|\geq \beta$. Since $\|u_{\lambda_{n}}\|\leq C$ for any $n\in \N$, we can proceed as in the proof of Lemma \ref{lemma4.3} to show that $u_{\lambda_{n}}\rightharpoonup u_{0}\neq 0$ in $H$. Then, by using Lemma \ref{lemma4.4}, we can see that 
$$
\I(u_{0})\leq \lim_{n\rightarrow \infty} \I(u_{\lambda_{n}})\leq c_{1}<\frac{s}{N} S_{*}^{\frac{N}{2s}}
\mbox{ and } \I'(u_{0})=0.
$$ 
Let us define
$$
m=\inf\{\I(u): u\in H, u\neq 0, \I'(u)=0\}.
$$
Being $\I'(u_{0})=0$, it is clear that $m\leq \I(u_{0})<\frac{s}{N} S_{*}^{\frac{N}{2s}}$. Now, by using the definition of $m$, we can find $(v_{n})\subset H$ such that $v_{n}\neq 0$, $\I(v_{n})\rightarrow m$ and $\I'(v_{n})=0$. Arguing as in (\ref{bound1}) and (\ref{bound2}), we can show that $(v_{n})$ is bounded in $H$, and that there exists $\beta>0$ independent of $n$ such that $\|v_{n}\|\geq \beta$. This means that $m>-\infty$. Proceeding similarly to the proof of Lemma \ref{lemma4.3}, we can see that $v_{n}\rightharpoonup v_{0}\neq 0$ in $H$.  Then, by using Lemma \ref{lemma4.4}, we can deduce that $\I'(v_{0})=0$ and $\I(v_{0})\leq m$. Since $\I'(v_{0})=0$, we also have that $\I(v_{0})\geq m$.
Therefore, we have proved that $v_{0}\neq 0$ is such that  $\I(v_{0})= m$ and $\I'(v_{0})=0$, that is $v_{0}$ is a ground state of (\ref{P}).
\end{proof}

\noindent {\bf Acknowledgements.}
The authors would like to express their sincere gratitude to the referee for careful reading the manuscript and valuable comments and suggestions.
The paper has been carried out under the auspices of the INdAM - GNAMPA Project 2017 titled: {\it Teoria e modelli per problemi non locali}.

\end{document}